\pdfoutput=1

\documentclass[10pt]{amsart}

\usepackage{stmaryrd}

\usepackage[colorlinks=true]{hyperref}
\usepackage{amsthm}
\theoremstyle{definition}
\newtheorem{definition}{Definition}
\theoremstyle{lemma}
\newtheorem{lemma}{Lemma}
\newtheorem*{remark}{Remark}
\theoremstyle{theorem}
\newtheorem{theorem}{Theorem}
\theoremstyle{assumption}
\newtheorem{assumption}{Assumption}

\usepackage{fullpage}
\usepackage{amsmath,amssymb,amsfonts}
\usepackage{array} 
\usepackage{mathtools}
\usepackage{bm}
\usepackage{bbm}
\usepackage{tikz}
\usepackage[normalem]{ulem}
\usepackage{hhline}
\usepackage{graphicx}
\usepackage{subfig}
\usepackage{color}
\usepackage{doi}

\usepackage{pgfplots}
\usepackage{pgfplotstable}
\definecolor{markercolor}{RGB}{124.9, 255, 160.65}
\pgfplotsset{
compat=1.3,
width=10cm,
tick label style={font=\small},
label style={font=\small},
legend style={font=\small}
}

\usetikzlibrary{calc}
\usetikzlibrary{intersections} 

\newcommand{\logLogSlopeTriangle}[5]
{

    \pgfplotsextra
    {
        \pgfkeysgetvalue{/pgfplots/xmin}{\xmin}
        \pgfkeysgetvalue{/pgfplots/xmax}{\xmax}
        \pgfkeysgetvalue{/pgfplots/ymin}{\ymin}
        \pgfkeysgetvalue{/pgfplots/ymax}{\ymax}

        \pgfmathsetmacro{\xArel}{#1}
        \pgfmathsetmacro{\yArel}{#3}
        \pgfmathsetmacro{\xBrel}{#1-#2}
        \pgfmathsetmacro{\yBrel}{\yArel}
        \pgfmathsetmacro{\xCrel}{\xArel}

        \pgfmathsetmacro{\lnxB}{\xmin*(1-(#1-#2))+\xmax*(#1-#2)} 
        \pgfmathsetmacro{\lnxA}{\xmin*(1-#1)+\xmax*#1} 
        \pgfmathsetmacro{\lnyA}{\ymin*(1-#3)+\ymax*#3} 
        \pgfmathsetmacro{\lnyC}{\lnyA+#4*(\lnxA-\lnxB)}
        \pgfmathsetmacro{\yCrel}{\lnyC-\ymin)/(\ymax-\ymin)} 

        \coordinate (A) at (rel axis cs:\xArel,\yArel);
        \coordinate (B) at (rel axis cs:\xBrel,\yBrel);
        \coordinate (C) at (rel axis cs:\xCrel,\yCrel);

        \draw[#5]   (A)-- node[pos=0.5,anchor=north] {}
                    (B)-- 
                    (C)-- node[pos=0.5,anchor=west] {#4}
                    cycle;
    }
}

\newcommand{\logLogSlopeTriangleFlip}[5]
{

    \pgfplotsextra
    {
        \pgfkeysgetvalue{/pgfplots/xmin}{\xmin}
        \pgfkeysgetvalue{/pgfplots/xmax}{\xmax}
        \pgfkeysgetvalue{/pgfplots/ymin}{\ymin}
        \pgfkeysgetvalue{/pgfplots/ymax}{\ymax}

        \pgfmathsetmacro{\xBrel}{#1-#2}
        \pgfmathsetmacro{\yBrel}{#3}
        \pgfmathsetmacro{\xCrel}{#1}

        \pgfmathsetmacro{\lnxB}{\xmin*(1-(#1-#2))+\xmax*(#1-#2)} 
        \pgfmathsetmacro{\lnxA}{\xmin*(1-#1)+\xmax*#1} 
        \pgfmathsetmacro{\lnyA}{\ymin*(1-#3)+\ymax*#3} 
        \pgfmathsetmacro{\lnyC}{\lnyA+#4*(\lnxA-\lnxB)}
        \pgfmathsetmacro{\yCrel}{\lnyC-\ymin)/(\ymax-\ymin)} 

	\pgfmathsetmacro{\xArel}{\xBrel}
        \pgfmathsetmacro{\yArel}{\yCrel}

        \coordinate (A) at (rel axis cs:\xArel,\yArel);
        \coordinate (B) at (rel axis cs:\xBrel,\yBrel);
        \coordinate (C) at (rel axis cs:\xCrel,\yCrel);

        \draw[#5]   (A)-- node[pos=0.5,anchor=east] {#4}
                    (B)-- 
                    (C)-- node[pos=0.5,anchor=south] {}
                    cycle;
    }
}

\DeclareMathOperator{\diam}{diam}

\renewcommand{\hat}{\widehat}
\renewcommand{\tilde}{\widetilde}
\newcommand{\td}[2]{\frac{{\rm d}#1}{{\rm d}{\rm #2}}}
\newcommand{\pd}[2]{\frac{\partial#1}{\partial#2}}

\newcommand{\nor}[1]{\left\| #1 \right\|}
\newcommand{\LRp}[1]{\left( #1 \right)}
\newcommand{\LRs}[1]{\left[ #1 \right]}
\newcommand{\LRa}[1]{\left\langle #1 \right\rangle}
\newcommand{\LRb}[1]{\left| #1 \right|}
\newcommand{\LRc}[1]{\left\{ #1 \right\}}

\newcommand{\jump}[1] {\ensuremath{\llbracket#1\rrbracket}}
\newcommand{\avg}[1] {\ensuremath{\LRc{\!\{#1\}\!}}}
\newcommand{\Grad} {\ensuremath{\nabla}}

\newcommand{\diag}[1]{{\rm diag}\LRp{#1}}

\newcommand*\diff[1]{\mathop{}\!{\mathrm{d}#1}} 

\date{}
\author{Jesse Chan, Lucas C. Wilcox}
\title{Discretely entropy stable weight-adjusted discontinuous Galerkin methods on curvilinear meshes}
\graphicspath{{./figs/}}

\begin{document}

\maketitle


\begin{abstract}
We construct entropy conservative and entropy stable high order accurate discontinuous Galerkin (DG) discretizations for time-dependent nonlinear hyperbolic conservation laws on curvilinear meshes.  The resulting schemes preserve a semi-discrete quadrature approximation of a continuous global entropy inequality.  The proof requires the satisfaction of a discrete geometric conservation law, which we enforce through an appropriate polynomial approximation.  We extend the construction of entropy conservative and entropy stable DG schemes to the case when high order accurate curvilinear mass matrices are approximated using low-storage weight-adjusted approximations, and describe how to retain global conservation properties under such an approximation.  The theoretical results are verified through numerical experiments for the compressible Euler equations on triangular and tetrahedral meshes.  
\end{abstract}



\section{Introduction}

High order discontinuous Galerkin (DG) methods are attractive for the simulation of time-dependent wave propagation due to their low numerical dispersion and dissipation \cite{hu1999analysis, ainsworth2004dispersive} and ability to handle unstructured meshes and complex geometries.  These same properties make them attractive for the resolution of transient waves and vortices in compressible flow \cite{wang2013high}.  However, whereas the construction of stable DG methods for wave propagation is relatively well-established, it is not possible to extend the same formulations directly from linear wave problems to the nonlinear conservation laws which govern compressible fluid flow.  

The low numerical dissipation of high order DG methods combined with the lack of inherently stable formulations for nonlinear conservation laws has given high order discretizations the reputation of being non-robust and highly sensitive to instabilities and under-resolved features \cite{wang2013high}.  This instability is addressed in practice by adding additional stabilization through limiting, filtering, or artificial viscosity \cite{persson2006sub, krivodonova2007limiters, barter2010shock, guermond2011entropy}.  However, these approaches are typically ad-hoc, and do not guarantee the stability of the resulting scheme.  Moreover, high order accuracy can be lost if the stabilization is too strong.  

The lack of inherently stable formulations was addressed for high order nodal DG methods on quadrilateral and hexahedral meshes in \cite{fisher2013high, carpenter2014entropy}.  The resulting schemes ensure that that the numerical solution satisfies a semi-discrete version of an entropy inequality, independently of discrete effects such as under-integration.  This results in significantly more robust high order simulations, where the numerical solution does not blow up even in the presence of under-resolved features such as shock discontinuities or turbulence.  High order entropy stable schemes have since been extended to staggered grid and non-conforming \cite{parsani2016entropy, friedrich2017entropy} tensor product elements, as well as to simplicial meshes \cite{chen2017entropy, chan2017discretely, crean2018entropy}.  

Entropy stable methods have largely relied on a finite difference summation-by-parts (SBP) framework.  The summation-by-parts property holds for spectral element DG methods (DG-SEM), which assume a polynomial basis which collocates the solution at Gauss--Legendre--Lobatto (GLL) quadrature nodes on tensor product elements.  The collocation approach requires the number of quadrature nodes to be identical to the number of basis functions.  Because analogous quadrature rules on the triangle and tetrahedron must contain more nodes than the dimension of the underlying polynomial space \cite{helenbrook2009existence, hicken2016multidimensional}, GLL-like collocation cannot be replicated on simplicial elements.  It is still possible to construct entropy stable schemes on simplicial elements within the SBP framework \cite{chen2017entropy, crean2018entropy}.   However, the resulting SBP operators are not ``modal'', in the sense that the matrices are not associated with an underlying basis or approximation space.  

Entropy stable schemes were extended to modal discretizations on affine meshes in \cite{chan2017discretely}, allowing for the use of arbitrary basis functions and over-integrated quadrature rules.  In this work, we show how to extend the construction of modal high order entropy stable schemes to curved meshes.  As noted in \cite{crean2018entropy}, this requires the use of a split formulation for the geometric terms involved in differentiation, as well as the satisfaction of a discrete geometric conservation law (GCL) \cite{thomas1979geometric, kopriva2006metric}.  We also show how to extend entropy stable schemes to accomodate weight-adjusted mass matrices, which provide low storage approximations of inverse weighted mass matrices appearing for curved meshes \cite{chan2016weight2}.  These weight-adjusted approximations can be applied more efficiently than inverse weighted mass matrices on many-core architectures such as Graphics Processing Units (GPUs) \cite{chan2017weight}.   However, additional steps are required to ensure entropy stability and the conservation of mass, momentum, and energy under weight-adjusted mass matrices.

The paper is organized as follows: Section~\ref{sec:1} reviews the derivation of entropy inequalities for systems of nonlinear conservation laws, and describes why this does not hold under a high order DG discretization.  Section~\ref{sec:2} describes the construction of entropy stable DG methods for curved meshes using weighted mass matrices, and Section~\ref{sec:3} describes how to extend this to the case of weight-adjusted mass matrices.  Section~\ref{sec:4} describes how to enforce the geometric conservation law using a modification of the approach described in \cite{kopriva2006metric}, and Section~\ref{sec:num} concludes by presenting two and three-dimensional numerical experiments which verify the accuracy and stability of the presented schemes.  

\section{Systems of nonlinear conservation laws}
\label{sec:1}
This work addresses high order accurate schemes for the following system of $n$ nonlinear conservation laws in $d$ dimensions 
\begin{equation}
  \pd{\bm{u}}{t} + \sum_{j=1}^d\pd{\bm{f}_j(\bm{u})}{x_j}  = 0, \qquad \bm{u} : \mathbb{R}^d \times [0,\infty) \rightarrow  \mathbb{R}^n, \qquad \bm{f}_j : \mathbb{R}^n\rightarrow \mathbb{R}^n,
\label{eq:nonlineqs}
\end{equation}
where $\bm{u}(\bm{x},t)$ denotes the \emph{conservative variables} for this system.  
We are interested in nonlinear conservation laws for which an entropy function $U(\bm{u})$ exists, where $U(\bm{u})$ is convex with respect to the conservative variables $\bm{u}$.  If this function exists, then it is possible to define \emph{entropy variables} $\bm{v}(\bm{u}) = \pd{U}{\bm{u}}$.  These functions symmetrize the system of nonlinear conservation laws (\ref{eq:nonlineqs}) \cite{hughes1986new}.  

It can be shown (see, for example, \cite{mock1980systems}) that symmetrization is equivalent to the existence of entropy flux functions $F_j(\bm{u})$ and entropy potentials $\psi_j$ such that
\[
\bm{v}^T \pd{\bm{f}_j}{\bm{u}} = \pd{F_j(\bm{u})}{\bm{u}}^T, \qquad \psi_j(\bm{v}) = \bm{v}^T\bm{f}_j(\bm{u}(\bm{v})) - F_j(\bm{u}(\bm{v})), \qquad \psi_j'(\bm{v}) = \bm{f}_j(\bm{u}(\bm{v})).
\]
Smooth solutions of (\ref{eq:nonlineqs}) can be shown to satisfy a conservation of entropy by multiplying (\ref{eq:nonlineqs}) by $\bm{v}(\bm{u})$. Using the definition of the entropy variables, entropy flux, and the chain rule yields
\begin{equation}
\bm{v}^T\pd{\bm{f}_j(\bm{u})}{x_j} = \pd{U(\bm{u})}{\bm{u}}^T\pd{\bm{f}_j(\bm{u})}{\bm{u}}\pd{\bm{u}}{x_j} = \pd{F_j(\bm{u})}{x_j}, \qquad \pd{U(\bm{u})}{t} + \sum_{j=1}^d \pd{F_j(\bm{u})}{x_j} = 0.
\label{eq:chainrule}
\end{equation}
%
Let $\Omega \subset \mathbb{R}^d$ be a closed domain with boundary $\partial \Omega$.  Integrating over $\Omega$ and using Gauss' theorem on the spatial derivative yields
\begin{equation}
  \int_{\Omega}\pd{U(\bm{u})}{t}\diff{\bm{x}} + \int_{\partial \Omega} \sum_{j=1}^d \LRp{\bm{v}{(\bm{u})}^T\bm{f}_j(\bm{u}) - \psi_j\LRp{\bm{v}(\bm{u})}}n_j \diff{\bm{x}} = 0,
\label{eq:entropyeq}
\end{equation}
where $\bm{n} = \LRp{n_1,\ldots,n_d}^T$ denotes the unit outward normal vector on $\partial \Omega$.  

General solutions (including non-smooth solutions such as shocks) satisfy an entropy \emph{inequality}
\begin{equation}
  \int_{\Omega}\pd{U(\bm{u})}{t}\diff{\bm{x}} + \int_{\partial \Omega} \sum_{j=1}^d \LRp{\bm{v}{(\bm{u})}^T\bm{f}_j(\bm{u}) - \psi_j\LRp{\bm{v}(\bm{u})}}n_j \diff{\bm{x}} \leq 0,
\label{eq:entropyineq}
\end{equation}
which results from considering solutions of an appropriate viscous form of the equations (\ref{eq:nonlineqs}) and taking the limit as viscosity vanishes.  In this work, schemes which satisfy a discrete form of (\ref{eq:entropyineq}) will be constructed by first enforcing a discrete version of entropy conservation (\ref{eq:entropyeq}), then adding an appropriate numerical dissipation which will enforce the entropy inequality (\ref{eq:entropyineq}).

\subsection{Standard DG formulations for nonlinear conservation laws}

We begin by reviewing the construction of standard high order accurate DG formulations for (\ref{eq:nonlineqs}).  

\subsubsection{Mathematical notation}

Let the domain $\Omega \subset \mathbb{R}^d$ be decomposed into elements (subdomains) $D^k$, and let $\hat{D}$ denote a $d$-dimensional reference element with boundary $\partial \hat{D}$.  Let $\hat{\bm{x}} = \LRc{\hat{x}_1,\ldots,\hat{x}_d}$ denote coordinates on $\hat{D}$, and let $\hat{n}_i$ denote and the $i$th component of the unit normal vector on $\partial \hat{D}$.  We assume that $\hat{n}_i$ is constant; i.e., that the faces of the reference element are planar (this assumption holds for all commonly used reference elements \cite{chan2015gpu}).  

We will assume that each physical element $D^k$ is the image of $\hat{D}$ under some smoothly differentiable mapping $\bm{\Phi}_k(\hat{\bm{x}})$ such that
\[
\bm{x} = \bm{\Phi}_k(\hat{\bm{x}}), \qquad \bm{x}\in D^k.
\]
This also implies that integrals over physical elements can be mapped back to the reference element as follows
\[
\int_{D^k} u \diff{\bm{x}} = \int_{\hat{D}} u J^k\diff{\hat{\bm{x}}}, 
\]
where $J^k$ denotes the determinant of the Jacobian of $\bm{\Phi}_k$.  Integrals over physical faces of $D^k$ can similarly be mapped back to reference faces.

We define an approximation space using degree $N$ polynomials on the reference element.  For example, on a $d$-dimensional reference simplex, the natural polynomial space are total degree $N$ polynomials 
\[
P^N\LRp{\widehat{D}} = \LRc{\hat{x}_1^{i_1}\ldots\hat{x}_d^{i_d}, \quad \hat{\bm{x}} \in \widehat{D}, \quad 0\leq \sum_{k=1}^d i_k \leq N}.
\]
Other element types possess different natural polynomial spaces \cite{chan2015gpu}, but typically contain the space of total degree $N$ polynomials.  This work is directly applicable to other elements and spaces as well.  We denote the dimension of the approximation space $P^N$ as $N_p = {\rm dim}\LRp{P^N\LRp{\widehat{D}}}$.  We also define trace spaces for each face of the reference element.  Let $\hat{f}$ be a face of the reference element $\hat{D}$.  The trace space over $\hat{f}$ is defined as the space of traces of functions in $P^N\LRp{\hat{D}}$
\[
P^N_f \LRp{\hat{f}} = \LRc{ \left.u\right|_{\hat{f}}, \quad u \in P^N\LRp{\hat{D}}}, \qquad \hat{f}\in \partial\hat{D}.
\]
We denote the dimension of the trace space as ${\rm dim}\LRp{P^N_f\LRp{\hat{f}}} = N^f_p$.

We next define the $L^2$ norm and inner products over the reference element $\hat{D}$ and the surface of the reference element $\partial \hat{D}$ as
\[
  \LRp{\bm{u},\bm{v}}_{\hat{D}} =  \int_{\widehat{D}} \bm{u}\cdot\bm{v} \diff{\hat{\bm{x}}}, \qquad \nor{\bm{u}}^2_{\hat{D}} = (\bm{u},\bm{u})_{\hat{D}}, \qquad \LRa{\bm{u},\bm{v}}_{\partial \hat{D}} = \int_{\partial \hat{D}} \bm{u} \cdot \bm{v} \diff{\hat{\bm{x}}}.
\]
We also introduce the continuous $L^2$ projection operator $\Pi_N$ and lifting operator $L$.  For $u \in L^2\LRp{\widehat{D}}$, the $L^2$ projection $\Pi_N u$ is defined through
\begin{equation}
\int_{\widehat{D}} \Pi_N u v \diff{\hat{\bm{x}}} = \int_{\widehat{D}} u v \diff{\hat{\bm{x}}}, \qquad \forall v\in P^N\LRp{\hat{D}}.
\label{eq:l2proj}
\end{equation}
Likewise, for a boundary function $u \in L^2\LRp{\partial \hat{D}}$, the lifting operator $L$ \cite{hesthaven2007nodal, di2011mathematical} is defined through 
\begin{equation}
\LRp{L u,v}_{\hat{D}} = \LRa{u,v}_{\partial \hat{D}}, \qquad \forall v \in P^N\LRp{\hat{D}}.
\label{eq:lift}
\end{equation}

Finally, we introduce $L^2, L^{\infty}$ Sobolev norms and spaces, which will be utilized for error estimates.  The $L^2$ space is defined as the space of functions with finite $L^2$ norm.  The Lebesgue $L^\infty$ norm and the associated $L^\infty$ space over a general domain $\Omega$ are 
\begin{align*}
\nor{u}_{L^{\infty}\LRp{\Omega}} &= \inf\LRc{C \geq 0: \LRb{u\LRp{\bm{x}}} \leq C \quad \forall \bm{x}\in \Omega}, \qquad
L^{\infty}\LRp{\Omega} = \LRc{u: \Omega\rightarrow \mathbb{R}, \quad \nor{u}_{L^{\infty}\LRp{\Omega}} < \infty}.
\end{align*}
The $L^2$ and $L^{\infty}$ Sobolev norms of degree $s$ are then defined  as
\begin{align*}
\nor{u}_{W^{s,2}\LRp{\Omega}}^2 &= {\sum_{\LRb{\alpha}\leq s} \nor{ D^{\alpha} u}_{L^2\LRp{\Omega}}^2}, \qquad \nor{u}_{W^{s,\infty}\LRp{\Omega}} = \max_{\LRb{\alpha}\leq s} \nor{D^{\alpha}u}_{L^{\infty}\LRp{\Omega}},
\end{align*}
respectively.  Here $\alpha = \LRc{\alpha_1,\ldots,\alpha_d}$ is a multi-index
of order $\LRb{\alpha} = \alpha_1 + \cdots + \alpha_d$ such that
\[
D^{\alpha}u = \pd{^{\alpha_1}}{x_1^{\alpha_1}}\cdots\pd{^{\alpha_d}}{x_d^{\alpha_d}} u.
\]
The Sobolev spaces $W^{s,2}$ and $W^{s,\infty}$ are defined as the spaces of functions with finite $L^2$ and $L^\infty$ Sobolev norms of degree $s$, respectively.

\subsubsection{Discontinuous Galerkin formulations and the $L^2$ projection}

Discontinuous Galerkin methods have been widely applied to systems of nonlinear conservation laws (\ref{eq:nonlineqs}) \cite{cockburn1989tvb, cockburn1998runge, cockburn2001devising}.  The development of new discontinuous Galerkin methods for nonlinear conservation laws has focused heavily on the choice of numerical flux \cite{qiu2006numerical} or the development of  slope limiters \cite{krivodonova2007limiters, zhang2012maximum} and artificial viscosity strategies \cite{persson2006sub, barter2010shock, klockner2011viscous}.  However, the treatment of the underlying volume discretization remains relatively unchanged between each of these approaches.  

Ignoring terms involving filters, limiters, or artificial viscosity, a semi-discrete ``weak'' DG formulation  for (\ref{eq:nonlineqs}) can be given locally over an element $D^k$: find $\bm{u}\in \LRp{P^N\LRp{D^k} \times [0,\infty)}^n$ such that
\begin{align}
\int_{D^k} \LRp{\pd{\bm{u}}{t}\cdot \bm{v} - \sum_{j=1}^d\bm{f}_j(\bm{u}) \cdot \pd{\bm{v}}{x_i}} \diff{\bm{x}} 
+ \sum_{j=1}^d \int_{\partial D^k} \LRp{\bm{f}^*_j\LRp{\bm{u}^+,\bm{u}} }\cdot \bm{v} n_j  \diff{\bm{x}} = 0, \qquad \forall \bm{v}\in \LRp{P^N\LRp{D^k}}^n,
\label{eq:weakdg}
\end{align}
where the numerical flux $\bm{f}^*$ is a function of the solution $\bm{u}$ on both $D^k$ and neighboring elements.  

Unfortunately, solutions to (\ref{eq:weakdg}) do not (in general) obey a discrete version of the entropy inequality (\ref{eq:entropyineq}).  Since (\ref{eq:entropyineq}) is a generalized statement of energy stability, the lack of a discrete entropy inequality implies that the discrete solution can blow up in finite time.  The reason for this is due to the fact that, in practice, the integrals in (\ref{eq:weakdg}) are not computed exactly and are instead approximated using polynomially exact quadratures.  This is compounded by the fact that the nonlinear flux function $\bm{f}_j\LRp{\bm{u}}$ is often rational and impossible to integrate exactly using polynomial quadratures.  

While this paper focuses on curved meshes, the inexactness of quadrature leads to the loss of the chain rule and thus the loss of entropy conservation or entropy dissipation (\ref{eq:entropyineq}) even on affine meshes.  We can rewrite (\ref{eq:weakdg}) in a strong form using a discrete quadrature-based $L^2$ projection.  For polynomial approximation spaces on affine meshes, $\pd{\bm{v}}{x_i}$ is polynomial.  
Then, mapping (\ref{eq:weakdg}) back to the reference element $\hat{D}$ and using the $L^2$ projection and (\ref{eq:l2proj}), we have that
\[
\int_{D^k} \bm{f}_j(\bm{u}) \cdot \pd{\bm{v}}{x_i} \diff{\bm{x}} = \int_{\hat{D}} \Pi_N \bm{f}_j(\bm{u}) \cdot \pd{\bm{v}}{x_i} J^k\diff{\bm{x}}.
\]
Thus, integrating by parts (\ref{eq:weakdg}) recovers a ``strong'' DG formulation involving the projection operator
\begin{align}
&\int_{D^k} \LRp{\pd{\bm{u}}{t} - \sum_{j=1}^d \pd{\Pi_N \bm{f}_j(\bm{u})}{x_j}} \cdot \bm{v} \diff{\bm{x}} \nonumber\\
&+ \sum_{j=1}^d \int_{\partial D^k} \LRp{\bm{f}^*_j\LRp{\bm{u}^+,\bm{u}} - \Pi_N\bm{f}_j(\bm{u})}\cdot \bm{v} n_j  \diff{\bm{x}} = 0, \qquad \forall \bm{v}\in \LRp{P^N\LRp{D^k}}^n.  
\label{eq:strongdg}
\end{align}
From this, we see that our discrete scheme does not differentiate the nonlinear flux function $\bm{f}_j\LRp{\bm{u}}$ exactly, but instead differentiates the projection of $\Pi_N \bm{f}_j\LRp{\bm{u}}$ onto polynomials of degree $N$.  Because the $L^2$ projection operator is introduced, the chain rule no longer holds at the discrete level and step (\ref{eq:chainrule}) of the proof of entropy conservation is no longer valid.  Thus, ensuring discrete entropy stability will require a discrete formulation of the system of nonlinear conservation laws (\ref{eq:nonlineqs}) from which we can prove a discrete entropy inequality without relying on the chain rule.  


\section{Discretely entropy stable DG methods on curved meshes}
\label{sec:2}
We will first show how to construct discretely entropy stable high order accurate DG methods on curvilinear meshes, but will present this using a matrix formulation as opposed to a continuous formulation.  This is to ensure that the effects of discretization, nonlinear, and quadrature are accounted for in the proof of semi-discrete entropy stability.  We first introduce quadrature-based matrices, which we will then use to construct discretely entropy stable DG formulations.

\subsection{Basis and quadrature rules}

{We now introduce quadrature-based matrices for the $d$-dimensional reference element $\widehat{D}$, which we will use to construct matrix-vector formulations of DG methods.   Assuming $u(\hat{\bm{x}}) \in P^N\LRp{\widehat{D}}$, it can be represented in terms of the vector of coefficients $\bm{u}$ using some polynomial basis $\phi_i$ of degree $N$ and dimension $N_p$ 
\[
  u(\hat{\bm{x}}) = \sum_{j=1}^{N_p}\bm{u}_j \phi_j(\widehat{\bm{x}}), \qquad P^N\LRp{\widehat{D}} = {\rm span}\LRc{\phi_i(\widehat{x})}_{i=1}^{N_p}.
\]

We construct quadrature-based matrices based on $\phi_i$ and appropriate volume and surface quadrature rules.  The volume and surface quadrature rules are given by points and positive weights $\LRc{(\hat{\bm{x}}_i, \hat{w}_i)}_{i=1}^{N_q}$ and $\LRc{(\hat{\bm{x}}^f_i, \hat{w}^f_i)}_{i=1}^{N^f_q}$, respectively.  We make the following assumptions on the strength of these quadratures: 
\begin{assumption}[Integration by parts under quadrature]
  The volume quadrature rule  $\LRc{(\hat{\bm{x}}_i, \hat{w}_i)}_{i=1}^{N_q}$ is exact for polynomials of degree $2N-1$.  Additionally, 
for any $u, v \in P^N\LRp{\hat{D}}$, integration by parts 
\[
  \LRp{\pd{u}{\hat{x}_i},v}_{\hat{D}} = \LRa{u,v\hat{n}_i}_{\partial \hat{D}} - \LRp{u,\pd{v}{\hat{x}_i}}_{\hat{D}}
\]
holds when the volume and surface integrals are approximated using quadrature.
\label{ass:quad}
\end{assumption}
Assumption~\ref{ass:quad} holds, for example, for any surface quadrature rule which is exact for degree $2N$ polynomials on the boundary of the reference element $\partial \hat{D}$.

\subsection{Reference element matrices}
\label{sec:matrix}

Let $\bm{W}, \bm{W}_f$ denote diagonal matrices whose entries are volume and surface quadrature weights, respectively.  The surface quadrature weights are given by quadrature weights on reference faces, which are mapped to faces of the reference element.  We define the volume and surface quadrature interpolation matrices $\bm{V}_q$ and $\bm{V}_f$ such that
\begin{align}
\LRp{\bm{V}_q}_{ij} &= \phi_j(\hat{\bm{x}}_i), \qquad 1 \leq j \leq N_p, \qquad 1 \leq i \leq N_q, \nonumber\\
\LRp{\bm{V}_f}_{ij} &= \phi_j(\hat{\bm{x}}^f_i), \qquad 1 \leq j \leq N_p, \qquad 1 \leq i \leq N^f_q,\label{eq:qinterp}
\end{align}
which map coefficients $\bm{u}$ to evaluations of $u$ at volume and surface quadrature points.  

Next, let ${\bm{D}}_i$ denote the differentiation matrix with respect to the $i$th coordinate, defined implicitly through the relations
\[
u(\hat{\bm{x}}) = \sum_{j=1}^{N_p} \bm{u}_j \phi_j(\hat{\bm{x}}), \qquad \pd{u}{\hat{\bm{x}}_i} = \sum_{j=1}^{N_p} \LRp{{\bm{D}}_i \bm{u}}_j\phi_j(\hat{\bm{x}}).
\]
The matrix ${\bm{D}}_i$ maps basis coefficients of some polynomial $u \in P^N\LRp{\hat{D}}$ to coefficients of its $i$th derivative with respect to the reference coordinate $\hat{\bm{x}}$, and is sometimes referred to as a ``modal'' differentiation matrix (with respect to a general non-nodal ``modal'' basis \cite{hicken2016multidimensional}).  

Using the volume quadrature interpolation matrix $\bm{V}_q$, we can compute a quadrature-based mass matrix $\bm{M}$ by evaluating $L^2$ inner products of different basis functions using quadrature
\[
  \bm{M} = \bm{V}_q^T\bm{W}\bm{V}_q, \qquad \bm{M}_{ij} = \sum_{k=1}^{N_q} \hat{w}_k \phi_j(\hat{\bm{x}}_k)\phi_i(\hat{\bm{x}}_k) \approx \int_{\hat{D}}\phi_j\phi_i \diff{\hat{\bm{x}}} = \LRp{\phi_j,\phi_i}_{\hat{D}}.
\]
The approximation in the formula for the mass matrix becomes an equality if the volume quadrature rule is exact for polynomials of degree $2N$.  The mass matrix is symmetric and positive definite under Assumption~\ref{ass:quad}; however, we do not make any distinctions between diagonal and dense (lumped) mass matrices in this work.

The mass matrix appears in the discretization of $L^2$ projection (\ref{eq:l2proj}) and lift operators (\ref{eq:lift}) using quadrature.  The result are quadrature-based $L^2$ projection and lift operators $\bm{P}_q, \bm{L}_q$, 
\begin{equation}
\bm{P}_q = \bm{M}^{-1}\bm{V}_q^T\bm{W}, \qquad \bm{L}_q = \bm{M}^{-1}\bm{V}_f^T \bm{W}_f,
\label{eq:projlift}
\end{equation}
which are discretizations of the continuous $L^2$ projection operator $\Pi_N$ and continuous lift operator $L$.  The matrix $\bm{P}_q$ maps a function (in terms of its evaluation at quadrature points) to coefficients of the $L^2$ projection in the basis $\phi_j(x)$, while the matrix $\bm{L}_q$ ``lifts'' a function (evaluated at surface quadrature points) from the boundary of an element to coefficients of a basis defined in the interior of the element.  

Finally, we introduce quadrature-based operators $\bm{D}_N^i$ which will be used to construct discretizations of our nonlinear conservation laws.  This operator was introduced in \cite{chan2017discretely} as a ``decoupled summation-by-parts'' operator
\begin{equation}
\bm{D}_N^i =\LRs{
\begin{array}{cc}
\bm{V}_q\bm{D}_i \bm{P}_q - \frac{1}{2}\bm{V}_q\bm{L}_q \diag{\hat{\bm{n}}_i\circ\bm{\hat{J}}_f}\bm{V}_f\bm{P}_q & \frac{1}{2}\bm{V}_q\bm{L}_q  \diag{\hat{\bm{n}}_i\circ\bm{\hat{J}}_f}\\
- \frac{1}{2}\diag{\hat{\bm{n}}_i\circ\bm{\hat{J}}_f} \bm{V}_f\bm{P}_q &  \frac{1}{2}\diag{\hat{\bm{n}}_i\circ\bm{\hat{J}}_f}
\end{array}
}
\label{eq:decoupledsbp}
\end{equation}
where $\hat{\bm{n}}_i$ is the vector containing values of the $i$th component of the unit normal on the surface of the reference element $\hat{D}$, and $\bm{\hat{J}}_f$ is the vector containing values of the face Jacobian factor $\hat{J}_f$ which result from mapping a face of $\hat{D}$ to a reference face.  Here $\hat{\bm{n}}_i\circ\bm{\hat{J}}_f$ is the Hadamard product (i.e., the entrywise product) of the vectors $\hat{\bm{n}}_i$ and $\bm{\hat{J}}_f$.  When combined with projection and lifting matrices, $\bm{D}_N^i$ produces a high order approximation of non-conservative products.  Let $\bm{f},\bm{g}$ denote vectors containing the evaluation of functions $f(\bm{x}),g(\bm{x})$ at both volume and surface quadrature points
\[
\LRs{\begin{array}{cc}\bm{P}_q & \bm{L}_q\end{array}} \diag{\bm{f}}\bm{D}_N^i \bm{g} \approx f\pd{g}{\hat{x}_i}.
\]
It was shown in \cite{chan2017discretely} that the matrix $\bm{D}_N^i$ satisfies several key properties.  First, it can be observed that $\bm{D}_N^i\bm{1} = 0$, where $\bm{1}$ is the vector of all ones.  Second, $\bm{D}_N^i$ satisfies a summation-by-parts property.  Let $\bm{Q}_N^i$ be the scaling of $\bm{D}_N^i$ by the diagonal matrix of volume and surface quadrature weights 
\[
\bm{Q}_N^i = \bm{W}_N \bm{D}_N^i, \qquad \bm{W}_N = \LRp{\begin{array}{cc}
\bm{W} &\\
& \bm{W}_f 
\end{array}}.
\]
Then, $\bm{Q}_N^i$ satisfies the following discrete analogue of integration by parts 
\begin{equation}
\bm{Q}_N^i + \LRp{\bm{Q}_N^i}^T = \bm{B}^i_N, \qquad \bm{B}_N = \LRp{\begin{array}{cc}
\bm{0}&\\
      & \bm{W}_f \diag{\hat{\bm{n}}_i\circ\hat{\bm{J}}_f}
\end{array}}.
\label{eq:sbpprop1}
\end{equation}
The matrix $\bm{D}_N^i$ reduces to polynomial differentiation when applied to polynomials, in the sense that
\begin{align}
\bm{D}_N^i \LRs{\begin{array}{c}
\bm{V}_q\\
\bm{V}_f
\end{array}} = \LRs{\begin{array}{c}
\bm{V}_q\bm{D}_i\\
\bm{0}
\end{array}}.
\label{eq:dnvqvf}
\end{align}


\subsection{Matrices on curved physical elements}
\label{sec:curv}

The key difference between curvilinear and affine meshes is that geometric terms now vary spatially over each element.  In practice, derivatives are computed over the reference element and mapped to the physical element $D^k$ through a change of variables formula
\[
J^k \pd{u}{x_i} = \sum_{j=1}^d G^k_{ij}\pd{u}{\hat{x}_j}, \qquad G^k_{ij} = J^k\pd{\hat{x}_j}{x_i},
\]
where we have defined the elements of the matrix $\bm{G}^k$ as the derivatives of the reference coordinates $\hat{x}_j$ with respect to the physical coordinates $x_i$ on $D^k$ times the Jacobian of the transformation from reference to physical coordinates $J^k$.  We denote evaluations of $G^k_{ij}$ at both volume and surface quadrature points as the vector $\bm{G}^k_{ij}$.

We assume in this work that the mesh is stationary.  It can be shown at the continuous level that, for any differentiable and invertible mapping, the quantity $\bm{G}^k$ satisfies a geometric conservation law (GCL) \cite{kopriva2006metric, thomas1979geometric}
\begin{equation}
\sum_{j=1}^d\pd{}{\hat{x}_j}G^k_{ij} = 0,
\label{eq:gcl}
\end{equation}
or that $\hat{\Grad}\cdot \bm{G}^k = 0$.  Using (\ref{eq:gcl}), the scaled physical derivative $J^k\pd{u}{x_i}$ can be computed via
\begin{equation}
J^k\pd{u}{x_i} = \frac{1}{2}\sum_{j=1}^d \LRp{G^k_{ij}\pd{u}{\hat{x}_j} + \pd{\LRp{G^k_{ij}u}}{\hat{x}_j}}.
\label{eq:splitderiv}
\end{equation}

We will require the following assumptions on the mesh, as well as the geometric terms and outward normal vectors:
\begin{assumption}[Mesh assumptions]
We assume that the mesh is quasi-uniform.  The mesh is also assumed to be watertight, such that normals are consistent across neighboring elements as follows: for a shared face $f$ between $D^k$ and $D^{k,+}$, the scaled outward normal vectors for each element are equal and opposite at all points such that 
\begin{align}
\bm{n}J^k_f = -\bm{n}^+J^{k,+}_f.
\label{eq:normalsign}
\end{align}
We also assume that the scaled matrix of geometric terms transforms scaled reference normal vectors to scaled physical normals, such that
\begin{align}
\sum_{j=1}^d G^k_{ij} \hat{n}_j\hat{J}_f = n_i J^k_f, \qquad \LRp{\sum_{j=1}^d \LRp{\bm{G}^k_{ij}}_f \circ\LRp{\hat{\bm{n}}_j\circ\hat{\bm{J}}_f}}
= \LRp{\bm{n}_i\circ\bm{J}^k_f},
\label{eq:normalconsistency}
\end{align}
where $\bm{n}_j$ and $\bm{J}^k_f$ are vectors containing evaluations of the physical unit normals and face Jacobian factors for $D^k$ at surface quadrature points, respectively.
Likewise, $\LRp{\bm{G}^k_{ij}}_f$ is a vector containing evaluations of $G^k_{ij}$ at the surface quadrature points.
\label{ass:norm}
\end{assumption}
The properties (\ref{eq:normalsign}) and (\ref{eq:normalconsistency}) hold at the continuous level for a watertight mesh \cite{ciarlet1978finite}, and thus at all points where the geometric terms are computed exactly.  However, we will also consider cases where the geometric terms $G^k_{ij}$ are modified to enforce a discrete form of (\ref{eq:gcl}); in these situations, it will be important to ensure that (\ref{eq:normalconsistency}) holds after such modifications.  

Similar to what is done to stabilize finite difference discretizations~\cite{nordstrom2006cfdf, gassner2016split}, we define physical differentiation matrices based on the approximation of (\ref{eq:splitderiv}).  Define $\bm{D}^i_k$ as
\[
  \bm{D}^i_k = \frac12 \sum_{j=1}^d \left(\diag{\bm{G}^k_{ij}}\bm{D}^j_N + \bm{D}^j_N\diag{\bm{G}^k_{ij}}\right).
\]
Using properties of the Hadamard product \cite{horn2012matrix}, we can rewrite $\bm{D}^i_k$ as 
\begin{equation}
\bm{D}^i_k = \sum_{j=1}^d \LRp{\bm{D}^j_N \circ \avg{\bm{G}^k_{ij}}}, \qquad \avg{\bm{G}^k_{ij}}_{mn} = \frac{1}{2}\LRp{\LRp{\bm{G}^k_{ij}}_m + \LRp{\bm{G}^k_{ij}}_n},
\label{eq:dik}
\end{equation}
where $\avg{\bm{G}^k_{ij}}$ denotes the matrix of averages between each of the entries of $\bm{G}^k_{ij}$.  From Assumption~\ref{ass:norm} and (\ref{eq:dik}), it is straightforward to show that (because $\avg{\bm{G}^k_{ij}}$ is symmetric) $\bm{Q}^i_k = \bm{W}_N\bm{D}^i_k$ also satisfies a summation-by-parts property
\begin{equation}
\bm{Q}^i_k + \LRp{\bm{Q}^i_k}^T = \bm{B}^i_k, \qquad \bm{B}^i_k = 
\LRp{\begin{array}{cc}
\bm{0}&\\
& \bm{W}_f \diag{\bm{n}_i\circ\bm{J}^k_f}
\end{array}}.
\label{eq:sbpk}
\end{equation}

Curvilinear mappings also imply that integrals over each physical element $D^k$ are no longer simple scalings of integrals over $\hat{D}$.  The $L^2$ projection of $u\in L^2\LRp{D^k}$ over a curvilinear element $D^k$ is defined through 
\begin{equation}
\LRp{\Pi^k_N u,v}_{D^k} = \LRp{u,v}_{D^k}, \qquad \forall v\in P^N\LRp{\hat{D}}.
\label{eq:l2curv}
\end{equation}
Mapping integrals to the reference element $\hat{D}$ yields
\begin{equation}
\LRp{\Pi^k_N u,v J^k}_{\hat{D}} = \LRp{u,vJ^k}_{\hat{D}}, \qquad \forall v\in P^N\LRp{\hat{D}}.
\label{eq:l2curvmap}
\end{equation}
For affine elements, $J^k$ is constant and can be cancelled.  Thus, the $L^2$ projection over affine elements is equivalent to simply taking the $L^2$ projection of a function over the reference element.  However, for curved elements, $J$ acts as a spatially varying weight within the $L^2$ inner product.  

Discretizing (\ref{eq:l2curvmap}) requires a weighted mass matrix.  We define a curved mass matrix over an element $D^k$ by weighting the discrete $L^2$ norm with values of $J$ at quadrature points
\begin{equation}
\bm{M}^k = \bm{V}_q^T \bm{W}\diag{\bm{J}^k}\bm{V}_q,
\label{eq:curvedmass}
\end{equation}
where $\bm{J}^k$ is a vector containing evaluation of the physical Jacobian factors for $D^k$ at volume quadrature points.
Then, curvilinear $L^2$ projection and lift matrices can be defined in a manner analogous to (\ref{eq:projlift})
\begin{equation}
\bm{P}^k_q = \LRp{\bm{M}^k}^{-1}\bm{V}_q^T\bm{W}\diag{\bm{J}^k}, \qquad \bm{L}^k_q = \LRp{\bm{M}^k}^{-1}\bm{V}_f^T\bm{W}_f\diag{\bm{J}^k_f}.
\label{eq:projliftcurved}
\end{equation}
These matrices are distinct from element to element, reflecting the fact that problem (\ref{eq:l2curvmap}) is distinct from element to element.  

\subsection{A discretely entropy stable DG formulation on curved meshes}

Given the matrices in Section~\ref{sec:curv}, we can now define a local entropy stable DG formulation on an element $D^k$.  Here, we seek an approximation solution $\bm{u}_N(\bm{x},t)$ to (\ref{eq:nonlineqs}), which is represented using vector-valued coefficients $\bm{u}_h(t)$ such that
\[
\bm{u}_N(\bm{x},t) = \sum_{j=1}^{N_p} \LRp{\bm{u}_h(t)}_j \phi_j(\bm{x}), \qquad \LRp{\bm{u}_h(t)}_j \in \mathbb{R}^n.
\]
Since the coefficients are vector valued, we assume that all matrices act component-wise on $\bm{u}_h$ in the Kronecker product sense.  

We first define the numerical fluxes $\bm{f}_{i,S}\LRp{\bm{u}_L,\bm{u}_R}$ as the bivariate function of ``left'' and ``right'' conservative variable states $\bm{u}_L, \bm{u}_R$.  
\begin{definition}
The numerical flux $\bm{f}_{i,S}\LRp{\bm{u}_L,\bm{u}_R}$ is entropy conservative (or entropy stable) if it satisfies the following conditions:
\begin{enumerate}
\item $\bm{f}_{i,S}\LRp{\bm{u}_L,\bm{u}_R} = \bm{f}_{i,S}\LRp{\bm{u}_R,\bm{u}_L}$ (symmetry).
\item $\bm{f}_{i,S}\LRp{\bm{u},\bm{u}} = \bm{f}_i\LRp{\bm{u}}$ (consistency).
\item $\bm{f}_{i,S}$ is referred to as entropy conservative if it satisfies conditions 1, 2, and  
\[
  \LRp{\bm{v}_L-\bm{v}_R}^T\bm{f}_{i,S}\LRp{\bm{u}_L,\bm{u}_R} = \psi_i\LRp{\bm{u}_L}-\psi_i\LRp{\bm{u}_R}.
\]
\end{enumerate}
\label{def:entropyflux}
\end{definition}

We now introduce the $L^2$ projection of the entropy variables $\bm{v}_h$ and the entropy-projected conservative variables $\tilde{\bm{u}}$ 
\begin{align}
\bm{u}_q = \bm{V}_q \bm{u}_h, \qquad \bm{v}_h = \bm{P}^k_q \bm{v}\LRp{\bm{u}_q}, \qquad 
\tilde{\bm{v}} = \LRs{\begin{array}{c}
\tilde{\bm{v}}_q\\
\tilde{\bm{v}}_f
\end{array}} = \LRs{\begin{array}{c}
\bm{V}_q\\
\bm{V}_f
\end{array}}\bm{v}_h, \qquad \tilde{\bm{u}} = \LRs{\begin{array}{c}
\tilde{\bm{u}}_q\\
\tilde{\bm{u}}_f
\end{array}} = \bm{u}\LRp{\tilde{\bm{v}}}.
\label{eq:evars1}
\end{align}
In (\ref{eq:evars1}), the entropy-projected conservative variables $\tilde{\bm{u}}$ denote the evaluation of the conservative variables in terms of the projected entropy variables at volume and face quadrature points.  We note that, under an appropriate choice of quadrature on quadrilaterals and hexahedra, this approach is equivalent to the approach taken in \cite{parsani2016entropy}, where the entropy variables are evaluated at Gauss nodes, then interpolated to a different set of nodes and used to compute the nonlinear fluxes.  

We now introduce a semi-discrete DG formulation for $\bm{u}_h$
\begin{align}
&\td{\bm{u}_h}{t} + \LRs{\begin{array}{cc}\bm{P}^k_q & \bm{L}^k_q\end{array}}
  \sum_{j=1}^d \LRp{2\bm{D}^j_k \circ \bm{F}_{j,S}}\bm{1} + \sum_{j=1}^d \bm{L}^k_q \diag{\bm{n}_j}\LRp{\bm{f}_j^* - \bm{f}_j(\tilde{\bm{u}}_f)} = 0,\label{eq:dgform1}\\
  &\LRp{\bm{F}_{j,S}}_{mn} = \bm{f}_{j,S}\LRp{\LRp{\tilde{\bm{u}}}_m,\LRp{\tilde{\bm{u}}}_n}, \qquad 1 \leq m,n \leq N_q + N^f_q,\nonumber\\
  &\bm{f}_j^* = \bm{f}_{j,S}(\tilde{\bm{u}}_f^+,\tilde{\bm{u}}_f) \text{ on interior faces},\nonumber
\end{align}
where $\tilde{\bm{u}}^+$ denotes the values of the entropy-projected conservative variables on the neighboring element across each face of $D^k$, and $\bm{f}_j^*$ on the boundary denotes the $j$th component of some numerical flux through which boundary conditions are imposed.  Note that the face/surface Jacobian factors $\bm{J}^k_f$ are incorporated into the definition of $\bm{L}^k_q$.  

Define the diagonal boundary quadrature matrix $\bm{W}_{\partial \Omega}$ such that
\[
\LRp{\bm{W}_{\partial \Omega} }_{ii} = \begin{cases}
  \bm{W}_f, & \text{if $\hat{\bm{x}}^f_i$ is on the $\partial \Omega$}\\
0, & \text{otherwise}. 
\end{cases}
\]
We have the following semi-discrete statement of entropy conservation:
\begin{theorem}
  Let $\bm{f}_{i,S}$ be an entropy conservative flux from Definition~\ref{def:entropyflux} and assume that $\bm{Q}^j_k\bm{1} = 0$ for $j = 1,\ldots,d$ over each element $D^k$.  Then, (\ref{eq:dgform1}) is entropy conservative in the sense that
\[
  \sum_k \bm{1}^T\diag{\bm{J}^k}\bm{W}\td{U(\bm{u}_q)}{t} = \sum_k \sum_{j=1}^d \bm{1}^T\diag{\bm{n}_j\circ\bm{J}^k_f}\bm{W}_{\partial \Omega} \LRp{\psi_j\LRp{\tilde{\bm{u}}_f}-\tilde{\bm{v}}_f^T\bm{f}_j^*}.
\]
\label{thm:stab1}
\end{theorem}
\begin{proof}
Under the assumption that $\bm{Q}^j_k\bm{1} = 0$ for $i = 1,\ldots,d$ over each element and (\ref{eq:sbpk}), the proof of entropy conservation is identical to that of \cite{chan2017discretely}.  
\end{proof}
An entropy stable scheme can be constructed by adding an entropy-dissipating penalty term, such as a Lax-Friedrichs penalization or the matrix dissipation terms introduced in \cite{chandrashekar2013kinetic, winters2017uniquely}.  For example, Lax-Friedrichs penalization can be incorporated by replacing the flux term with
\[
\bm{L}^k_q \diag{\bm{n}_j} \LRp{\bm{f}_j^*-\bm{f}(\bm{u})} \Longrightarrow \bm{L}^k_q\LRp{\diag{\bm{n}_j} \LRp{\bm{f}_j^*- \bm{f}(\bm{u})} - \frac{\lambda}{2}\jump{\tilde{\bm{u}}_f}},
\]
where $\lambda$ is an estimate of the maximum eigenvalue of $\pd{\bm{f}(\bm{u})}{\bm{u}}$ \cite{chen2017entropy, chan2017discretely}.


\section{Discretely stable and low storage DG methods on curved meshes}
\label{sec:3}
A disadvantage of the formulation (\ref{eq:dgform1}) is high storage costs, especially at high orders of approximation.  While the matrices $\bm{Q}^i_k$ can be applied to a vector without needing to explicitly store the matrix, the projection and lifting matrices (\ref{eq:projliftcurved}) differ from element to element, necessitating either explicit pre-computation and storage or the assembly and inversion of a weighted mass matrix for each right hand side evaluation.  The latter option is computationally expensive, while the former option increases storage costs.  This increase in storage can result in suboptimal performance on modern computational architectures \cite{chan2017weight}, due to the increasing cost of memory operations and data movement compared to arithmetic operations.  

In this section, we present a discretely entropy stable scheme which avoids this high storage cost through the use of a low-storage weight-adjusted approximation to the inverse of a weighted mass matrix.  To ensure a discrete entropy conservation or a discrete entropy inequality, we also modify the formulation (\ref{eq:dgform1}) to take into account the use of a weight-adjusted mass matrix.

\subsection{A weight-adjusted approximation to the curvilinear mass matrix}

The presence of the weighted $L^2$ inner product $\LRp{J^k \Pi_N^k u,v}_{\hat{D}}$ in (\ref{eq:l2curvmap}) results in the presence of a weighted mass matrix.  Because the weight $J^k$ varies spatially over each element, the inverse of a weighted mass matrix is no longer a scaling of the inverse reference mass matrix.  The motivation for the weight-adjusted mass matrix is to replace the inversion of weighted mass matrices over each element with the application of inverse reference mass matrices and quadrature-based operations involving the spatially varying weights $J^k$ \cite{chan2016weight1, chan2016weight2}.  

To define a weight-adjusted approximation to the curvilinear $L^2$ inner product, we first define the operator $T_{w}^{-1}: L^2\rightarrow P^N$ as follows
\begin{equation}
  \LRp{wT_{w}^{-1} u,v}_{\hat{D}} = \LRp{ u,v}_{\hat{D}}, \qquad \forall v\in P^N\LRp{\hat{D}}.
\label{eq:wadgTw}
\end{equation}
Roughly speaking, $T_{w} u$ approximates $u/w$.  Thus, taking $w = 1/J^k$ provides an approximation of the curvilinear $L^2$ inner product
\begin{equation*}
\LRp{J^k u,v}_{\hat{D}} \approx \LRp{T_{1/J^k}^{-1} u,v}_{\hat{D}}.
\end{equation*}

Computing $T_{1/J^k}^{-1}u$ requires solving (\ref{eq:wadgTw}).  Let $u \in P^N\LRp{\hat{D}}$, and let $\bm{u}_J$ denote coefficients for the polynomial $T_{1/J^k}^{-1}u$.  This results in the following matrix system
\[
\bm{M}_{1/J^k}\bm{u}_{J} = \bm{M}\bm{u}, \qquad {\bm{M}_{1/J^k}} = \bm{V}_q^T \bm{W}\diag{1/\bm{J}^k}\bm{V}_q,
\]
which implies that, when restricted to polynomials, the matrix form of $T_{1/J^k}^{-1}$ is $\bm{M}_{1/J^k}^{-1}\bm{M}$.  Then, the weight-adjusted mass matrix is the Gram matrix with respect to the weight-adjusted inner product $\LRp{T_{1/J^k}^{-1} u,v}_{\hat{D}}$, such that
\[
\bm{M}^k \approx \bm{M}\bm{M}_{1/J^k}^{-1}\bm{M}, \qquad \LRp{\bm{M}^k}^{-1} \approx \bm{M}^{-1}\bm{M}_{1/J^k}\bm{M}^{-1}.
\]
The inverse of the weight-adjusted mass matrix can be applied in a matrix-free fashion by using quadrature to form $\bm{M}_{1/J^k} $.  This requires storage of the inverse reference mass matrix and the values of $\bm{J}^k$ at quadrature points.  Assuming that the number of quadrature points scales as $O(N^d)$ in $d$ dimensions, this yields a storage cost of $O(N^d)$ per-element compared to an $O(N^{2d})$ per element storage cost required for the storage of inverse weighted mass matrices $\LRp{\bm{M}^k}^{-1}$.  This application of the weight-adjusted mass matrix is typically applied using the $L^2$ projection matrix $\bm{P}_q$ as follows
\[
\bm{M}^{-1}\bm{M}_{1/J^k}\bm{M}^{-1} = \bm{P}_q \diag{1/\bm{J}^k} \bm{V}_q \bm{M}^{-1}.
\]
When evaluating the right hand side of a semi-discrete formulation such as (\ref{eq:dgform1}), the inverse mass matrix is typically merged into operations on the right hand side, such that the main work in applying the weight-adjusted mass matrix consists of applying the interpolation matrix $\bm{V}_q$, scaling by pointwise values of $1/\bm{J}^k$ at quadrature points, and multiplying by the $L^2$ projection matrix $\bm{P}_q$.

\subsection{A discretely entropy stable low storage DG formulation on curved meshes}

Given the weight-adjusted inverse mass matrix, we can also define a weight-adjusted version of the $L^2$ projection over a curved element $D^k$.  We refer to this operator as $\tilde{\Pi}^k_N: L^2\rightarrow P^N$, which satisfies
\[
\LRp{T^{-1}_{1/J^k}\tilde{\Pi}^k_N u,v}_{\hat{D}} = \LRp{uJ^k,v}_{\hat{D}}, \qquad \forall v\in P^N\LRp{\hat{D}}.
\]
It was shown in \cite{chan2016weight2} that $\tilde{\Pi}^k_N $ is given explicitly by
\begin{equation}
\tilde{\Pi}^k_N u = \Pi_N\LRp{\frac{1}{J^k}\Pi_N\LRp{uJ^k}},
\label{eq:wadgprojop}
\end{equation}
where $\Pi_N$ is the $L^2$ projection operator on the reference element $\hat{D}$.  We can discretize $\tilde{\Pi}^k_N$ using quadrature to yield a weight-adjusted projection matrix $\tilde{\bm{P}}^k_q$ 
\begin{align}
\tilde{\bm{P}}^k_q &= \bm{M}^{-1}\bm{M}_{1/\bm{J}^k}\bm{M}^{-1}\bm{V}_q^T\bm{W}\diag{\bm{J}^k} = \bm{M}^{-1}\bm{V}_q^T\bm{W}\diag{1/\bm{J}^k} \bm{V}_q\bm{P}_q\diag{\bm{J}^k} \nonumber\\
&= \bm{P}_q \diag{{1}/{\bm{J}^k}} \bm{V}_q\bm{P}_q \diag{\bm{J}^k}.
\label{eq:wadgproj}
\end{align}
We can similarly define a weight-adjusted lifting matrix $\tilde{\bm{L}}_q$ by replacing the weighted mass matrix in (\ref{eq:projliftcurved}) with the weight-adjusted mass matrix
\begin{align}
\tilde{\bm{L}}^k_q &= \bm{M}^{-1}\bm{M}_{1/\bm{J}^k}\bm{M}^{-1}\bm{V}_f^T\bm{W}_f\diag{\bm{J}^k_f} = \bm{M}^{-1}\bm{V}_q^T\bm{W}\diag{1/\bm{J}^k} \bm{V}_q\bm{L}_q\diag{\bm{J}^k_f} \nonumber\\
&= \bm{P}_q \diag{{1}/{\bm{J}^k}} \bm{V}_q\bm{L}_q \diag{\bm{J}^k_f}.
\label{eq:wadglift}
\end{align}

We can now introduce the weight-adjusted projection of the entropy variables $\bm{v}_h$ and the corresponding entropy-projected conservative variables $\tilde{\bm{u}}$ 
\begin{align}
\bm{u}_q = \bm{V}_q \bm{u}_h, \qquad \bm{v}_h = \tilde{\bm{P}}^k_q \bm{v}\LRp{\bm{u}_q}, \qquad 
\tilde{\bm{v}} = \LRs{\begin{array}{c}
\bm{V}_q\\
\bm{V}_f
\end{array}}\bm{v}_h, \qquad \tilde{\bm{u}} =  \LRs{\begin{array}{c}
\tilde{\bm{u}}_q\\
\tilde{\bm{u}}_f
\end{array}} = \bm{u}\LRp{\tilde{\bm{v}}}.
\label{eq:evars2}
\end{align}
A semi-discrete DG formulation for $\bm{u}_h$ can be constructed using the variables defined in (\ref{eq:evars2})
\begin{align}
&\td{\bm{u}_h}{t} + \LRs{\begin{array}{cc}\tilde{\bm{P}}^k_q & \tilde{\bm{L}}^k_q\end{array}}
  \sum_{j=1}^d \LRp{2\bm{D}^j_k \circ \bm{F}_{j,S}}\bm{1} + \sum_{j=1}^d \tilde{\bm{L}}^k_q \diag{\bm{n}_j}\LRp{\bm{f}_j^* - \bm{f}_j(\tilde{\bm{u}}_f)} = 0,\label{eq:dgform2}\\
  &\LRp{\bm{F}_{j,S}}_{mn} = \bm{f}_{j,S}\LRp{\LRp{\tilde{\bm{u}}}_m,\LRp{\tilde{\bm{u}}}_n}, \qquad 1 \leq m,n \leq N_q + N^f_q,\nonumber\\
  &\bm{f}_j^* = \bm{f}_{j,S}(\tilde{\bm{u}}_f^+,\tilde{\bm{u}}_f) \text{ on interior faces}.\nonumber
\end{align}
Since the  weight-adjusted mass matrix inverse is low-storage, and since the matrices $\bm{D}^j_k$ in (\ref{eq:dik}) can be assembled from reference matrices $\bm{D}^i_N$ and the values of geometric terms at quadrature points, the overall scheme requires only $O(N^d)$ storage per element.  We can additionally show that formulation (\ref{eq:dgform2}) is entropy conservative in the same sense as Theorem~\ref{thm:stab1}: 
\begin{theorem}
  Let $\bm{f}_{i,S}$ be an entropy conservative flux from Definition~\ref{def:entropyflux} and assume that $\bm{Q}^j_k\bm{1} = 0$ for $j = 1,\ldots,d$ over each element $D^k$.  Then, (\ref{eq:dgform2}) is entropy conservative in the sense that
\[
  \sum_k \bm{1}^T\diag{\bm{J}^k}\bm{W}\td{U(\bm{u}_q)}{t} = \sum_k \sum_{j=1}^d \bm{1}^T\diag{\bm{n}_j\circ\bm{J}^k_f}\bm{W}_{\partial \Omega} \LRp{\psi_j\LRp{\tilde{\bm{u}}_f}-\tilde{\bm{v}}_f^T\bm{f}_j^*}.
\]
\label{thm:stab2}
\end{theorem}
\begin{proof}
The proof is similar to that of Theorem~\ref{thm:stab1}.  
First, recall from the definitions of the weight-adjusted projection matrix~(\ref{eq:wadgproj}) and weight-adjusted lift matrix~(\ref{eq:wadglift}) that
\begin{align}
\tilde{\bm{P}}^k_q &= \bm{M}^{-1}\bm{M}_{1/J^k} \bm{M}^{-1} \bm{V}_q^T \bm{W}  \diag{\bm{J}^k  }\label{eq:projidentity},\\
\tilde{\bm{L}}^k_q &= \bm{M}^{-1}\bm{M}_{1/J^k} \bm{M}^{-1} \bm{V}_f^T \bm{W}_f\diag{\bm{J}^k_f}\label{eq:liftidentity}.
\end{align}
Then, multiplying by the weight-adjusted mass matrix $\bm{M}\bm{M}_{1/J^k}^{-1}\bm{M}$ on both sides of (\ref{eq:dgform2}) and using (\ref{eq:projidentity}), (\ref{eq:liftidentity}) yields the weak form of (\ref{eq:dgform2})
\begin{equation}
\bm{M}\bm{M}_{1/J^k}^{-1}\bm{M}\td{\bm{u}_h}{t} + \LRs{\begin{array}{cc}{\bm{V}}_q^T & {\bm{V}}_f^T\end{array}}
\sum_{j=1}^d \LRp{2\bm{Q}^j_k \circ \bm{F}_{j,S}}\bm{1} + \sum_{j=1}^d {\bm{V}}_f^T\bm{W}_f \diag{\bm{n}_j}\LRp{\bm{f}_j^* - \bm{f}_j(\tilde{\bm{u}}_f)} = 0.
\label{eq:dgform2weak}
\end{equation}
Testing with the weight-adjusted projection of the entropy variables $\bm{v}_h = \tilde{\bm{P}}^k_q \bm{v}\LRp{\bm{u}_q}$ and using (\ref{eq:projidentity}) then yields for the time term
\begin{align*}
\LRp{\tilde{\bm{P}}^k_q \bm{v}\LRp{\bm{u}_q}}^T\bm{M}\bm{M}_{1/J^k}^{-1}\bm{M}\td{\bm{u}_h}{t} &= \bm{v}\LRp{\bm{u}_q}^T\bm{W}\diag{\bm{J}^k}\bm{V}_q \bm{M}^{-1}\bm{M}_{1/J^k} \bm{M}^{-1} \bm{M}\bm{M}_{1/J^k}^{-1}\bm{M}\td{\bm{u}_h}{t}\\
                                                                                               &= \bm{v}\LRp{\bm{u}_q}^T\bm{W}\diag{\bm{J}^k} \td{\bm{V}_q\bm{u}_h}{t} = \bm{1}^T\bm{W}\diag{\bm{J}^k} \left({\rm diag}\LRp{\bm{v}\LRp{\bm{u}_q}}\td{\bm{u}_q}{t}\right) \\
                                                                                               &= \bm{1}^T\bm{W}\diag{\bm{J}^k} \td{U(\bm{u}_q)}{t}.
\end{align*}
The remainder of the proof is identical to that of Theorem~\ref{thm:stab1} and \cite{chan2017discretely}.
\end{proof}

\subsection{Analysis of weight-adjusted projection} 

The construction of the discretely entropy stable weight-adjusted DG formulation (\ref{eq:dgform2}) replaces the $L^2$ projection operator $\Pi_N^k$ with the weight-adjusted projection operator $\tilde{\Pi}^k_N$.
While this preserves entropy stability, it is unclear whether $\tilde{\Pi}^k_N$ is high order accurate.  In this section, we prove that the weight-adjusted projection is high order accurate due to the fact that, for a fixed geometric mapping and sufficiently regular $u$, the difference between the $L^2$ and weight-adjusted projection is $\nor{\Pi_N^k u - \tilde{\Pi}^k_N  u}_{L^2\LRp{\Omega}} = O(h^{N+2})$.  Because the approximation error for the $L^2$ projection is $O(h^{N+1})$ for sufficiently regular $u$, the difference between the $L^2$ and weight-adjusted projection converges faster than the $L^2$ best approximation error.  Consequentially, solutions computed using the $L^2$ and weight-adjusted projection are typically indistinguishable for a fixed geometric mapping \cite{chan2018multi}.  

We first note that $\tilde{\Pi}^k_N  = \Pi_N\LRp{\frac{1}{J^k}\Pi_N\LRp{uJ^k}}$ is self-adjoint with respect to the $J$-weighted $L^2$ inner product
\begin{equation}
  \LRp{J^k \tilde{\Pi}^k_N  u, v}_{\hat{D}} = \LRp{\Pi_N\LRp{\frac{1}{J^k}\Pi_N\LRp{uJ^k}}, vJ^k}_{\hat{D}} = \LRp{uJ^k, \Pi_N\LRp{\frac{1}{J^k}\Pi_N\LRp{vJ^k}}}_{\hat{D}} =  \LRp{uJ^k, \tilde{\Pi}^k_N  v}_{\hat{D}}.
\label{eq:PNsym}
\end{equation}
Furthermore, using that the operator $T_{1/J^k}^{-1}$ is self-adjoint for $v \in P^N\LRp{\hat{D}}$ with respect to the $L^2$ inner product \cite{chan2016weight1}, we find that a projection-like property holds for the weight-adjusted $L^2$ inner product
\begin{equation}
  \LRp{T_{1/J^k}^{-1} \tilde{\Pi}^k_N  u,v}_{\hat{D}} = \LRp{ \frac{1}{J^k}\Pi_N(uJ^k),T_{1/J^k}^{-1}v}_{\hat{D}} = \LRp{\Pi_N(u J^k),v}_{\hat{D}} = \LRp{u J^k,v}_{\hat{D}}, \qquad \forall v\in P^N{\hat{D}}.
\label{eq:PNproj}
\end{equation}

To prove $\nor{\Pi_N^k u - \tilde{\Pi}^k_N  u}_{L^2\LRp{\hat{D}}} = O(h^{N+2})$, we use a generalized inverse inequality and results from \cite{chan2016weight1,chan2016weight2}.  We first introduce a modification of Theorem 3.1 in \cite{warburton2013low, chan2016weight1}
\begin{theorem}
Let $D^k$ be a quasi-regular element with representative size $h = \diam\LRp{D^k}$, and let $\hat{D}$ be the reference element.  For $N \geq 0$, $w\in W^{N+1,\infty}\LRp{D^k}$, and $u\in W^{r,2}\LRp{D^k}$, 
\begin{align*}
\nor{u - \frac{1}{w} \Pi_N\LRp{{u}{w}}}_{L^2\LRp{\hat{D}}} &\leq C h^{\min\LRp{r,N+1}} \nor{\frac{1}{\sqrt{J^k}}}_{L^{\infty}\LRp{D^k}}\nor{\frac{1}{w}}_{L^{\infty}\LRp{D^k}} \nor{w}_{W^{N+1,\infty}\LRp{D^k}} \nor{u}_{W^{r,2}\LRp{D^k}}.
\end{align*}
\label{thm:wproj}
\end{theorem}
The proof is a straightforward modification of the proofs presented in \cite{warburton2013low, chan2016weight1} accounting for reduced regularity of $u$ when $r < (N+1)$.  The next result we need is a generalized inverse inequality.  
\begin{lemma}
\label{lemma:sobolev}
Let $v \in P^N\LRp{\hat{D}}$, and let $h = \diam\LRp{D^k}$.  Then,
\[
  \nor{v}_{W^{N+1,2}\LRp{D^k}} \leq C_{N}  h^{-N} \nor{\sqrt{J^k}}_{L^{\infty}} \nor{\frac{1}{\sqrt{J^k}}}_{L^{\infty}} \nor{\frac{1}{J^k}\bm{G}^k}_{W^{N,\infty}\LRp{D^k}} \nor{v}_{L^2\LRp{D^k}}.
\]
where $C_{N}$ depends on $N$, but is independent of $h$.
\end{lemma}
\begin{proof}
  By applying Fa\`{a} di Bruno's formula, we can express the degree $(N+1)$ Sobolev norm of $v$ on $D^k$ in terms of derivatives of $v$ on the reference element $\hat{D}$.  Noting that all $(N+1)$ derivatives of $v$ disappear for $v\in P^N\LRp{\hat{D}}$ allows us to bound the degree $(N+1)$ Sobolev norm of $v$ by its degree $N$ Sobolev norm
\[
  \nor{v}_{W^{N+1,2}\LRp{D^k}} \leq C_N \nor{\frac{1}{J^k}\bm{G}^k}_{W^{N,\infty}\LRp{D^k}} \nor{v}_{W^{N,2}\LRp{D^k}},
\]
where $\bm{G}^k$ is the matrix of scaled geometric terms for $D^k$.  Then, a scaling argument \cite{ciarlet1978finite, brenner2007mathematical} yields
\begin{align}
\nor{v}_{W^{N,2}\LRp{D^k}}  \leq C_1 h^{-N} \nor{\sqrt{J^k}}_{L^{\infty}} \nor{v}_{W^{N,2}\LRp{\hat{D}}}.  
\end{align}
The quantity $\nor{v}_{W^{N,2}\LRp{\hat{D}}}$ can be bounded by noting that $v\in P^N\LRp{\hat{D}}$.  Since $P^N\LRp{\hat{D}}$ is finite-dimensional, the Sobolev norm can be bounded from above by the $L^2$ norm of $\hat{D}$ with a constant $C_{2}$ depending on $N, d$.  By another scaling argument, we have 
\begin{align}
\nor{v}_{W^{N+1,2}\LRp{\hat{D}}} \leq C_{2} \nor{v}_{L^2\LRp{\hat{D}}} \leq \nor{\frac{1}{\sqrt{J^k}}}_{L^{\infty}} C_{2}  \nor{v}_{L^2\LRp{D^k}}.
\end{align}
\end{proof}

We can now prove that $\tilde{\Pi}^k_N u$ is superconvergent to the curvilinear $L^2$ projection $\Pi_N^k u$: 
\begin{theorem}
Let $u \in W^{r,2}\LRp{D^k}$.  The difference between the $L^2$ projection $\Pi^k_Nu$ and the weight-adjusted projection $\tilde{\Pi}^k_N u$ is
\[
\nor{\Pi_N^k u - \tilde{\Pi}^k_N  u}_{L^2\LRp{D^k}} \leq C_N C_J h^{\min\LRp{r,N+1}+1}\nor{u}_{W^{N+1,2}\LRp{D^k}},
\]
where $C_N$ is a mesh-independent constant which depends on $N, d$ and $C_J$ is
\[
  C_J = \nor{J^k}_{L^{\infty}\LRp{D^k}}^{1.5}  \nor{\frac{1}{J^k}}_{L^{\infty}\LRp{D^k}}^{1.5} \nor{\frac{1}{J^k}}_{W^{N+1,\infty}\LRp{D^k}}\nor{J^k}_{W^{N+1,\infty}\LRp{D^k}}\nor{\frac{1}{J^k}\bm{G}^k}_{W^{N,\infty}\LRp{D^k}}.
\]
\label{thm:superconverge}
\end{theorem}
\begin{proof}
We can rewrite the norm of the difference between the weight-adjusted and $L^2$ projections
\[
\nor{\Pi_N^k  u - \tilde{\Pi}^k_N  u}_{L^2\LRp{D^k}}^2 = \LRp{\Pi_N^k  u - \tilde{\Pi}^k_N  u,vJ^k}_{\hat{D}}, \qquad v = \Pi_N^k  u - \tilde{\Pi}^k_N  u.
\]
Because $v \in P^N\LRp{\hat{D}}$, we can also evaluate the squared error as
\begin{align*}
\nor{\Pi_N^k  u - \tilde{\Pi}^k_N  u}_{L^2\LRp{D^k}}^2 &= \LRb{\LRp{\Pi_N^k u,vJ^k}_{\hat{D}} - \LRp{\tilde{\Pi}^k_N  u,vJ^k}_{\hat{D}}} = \LRb{\LRp{u,vJ^k}_{\hat{D}} - \LRp{\tilde{\Pi}^k_N  u,vJ^k}_{\hat{D}}} \\
&\leq \nor{J^k}_{L^{\infty}\LRp{D^k}} \LRb{\LRp{u-\tilde{\Pi}^k_Nu,v}_{\hat{D}}}.
\end{align*}
We can then note that $\tilde{\Pi}^k_Nu = \Pi_N\LRp{\frac{1}{J^k} \Pi_N\LRp{uJ^k}}$ to show that
\begin{align*}
{\LRp{u-\tilde{\Pi}^k_Nu,v}_{\hat{D}}} &= {\LRp{uJ^k,\frac{v}{J^k}}_{\hat{D}}-\LRp{\Pi_N\LRp{\frac{1}{J^k} \Pi_N\LRp{uJ^k}},v}_{\hat{D}}} \\
&= \LRp{uJ^k,\frac{v}{J^k}}_{\hat{D}}-\LRp{\Pi_N\LRp{uJ^k},\frac{v}{J^k}}_{\hat{D}}. 
\end{align*}
Adding and subtracting $\LRp{\Pi_N\LRp{uJ^k},\Pi_N\LRp{\frac{v}{J^k}}}_{\hat{D}}$ and using Theorem~\ref{thm:wproj} (noting that $v\in W^{N+1,2}\LRp{D^k}$ since it is polynomial) gives
\begin{align*}
\nor{\Pi^k_Nu-\tilde{\Pi}^k_Nu}_{L^2\LRp{D^k}}^2 &\leq \nor{J^k}_{L^{\infty}\LRp{D^k}}\LRb{\LRp{u-\tilde{\Pi}^k_Nu,v}_{\hat{D}}}\\
&= \nor{J^k}_{L^{\infty}\LRp{D^k}}\LRb{\LRp{uJ^k - \Pi_N\LRp{uJ^k},\frac{v}{J^k}-\Pi_N\LRp{\frac{v}{J^k}}}_{\hat{D}}} \\
&\leq \nor{J^k}_{L^{\infty}\LRp{D^k}}\nor{uJ^k - \Pi_N\LRp{uJ^k}}_{\hat{D}}\nor{\frac{v}{J^k}-\Pi_N\LRp{\frac{v}{J^k}}}_{\hat{D}}\\
&\leq C h^{\min\LRp{r,N+1}+N+1} \tilde{C}_J \nor{u}_{W^{r,2}\LRp{D^k}} \nor{v}_{W^{N+1,2}\LRp{D^k}},
\end{align*}
where 
\[
\tilde{C}_J = \nor{J^k}_{L^{\infty}\LRp{D^k}}\nor{\frac{1}{J^k}}_{L^{\infty}\LRp{D^k}} \nor{J^k}_{W^{N+1,\infty}\LRp{D^k}}\nor{\frac{1}{J^k}}_{W^{N+1,\infty}\LRp{D^k}}.
\]
Applying Lemma~\ref{lemma:sobolev} then yields
\begin{align*}
\nor{\Pi^k_Nu-\tilde{\Pi}^k_Nu}_{L^2\LRp{D^k}}^2 &\leq C_N h^{\min\LRp{r,N+1}+1} C_J \nor{u}_{W^{N+1,2}\LRp{D^k}} \nor{v}_{L^2\LRp{D^k}}.
\end{align*}
Dividing through by $\nor{v}_{L^2\LRp{D^k}} = \nor{\Pi^k_N u - \tilde{\Pi}^k_N  u}_{L^2\LRp{D^k}}$ gives the desired result.  
\end{proof}
Theorem~\ref{thm:wproj} can be used to show that the $L^2$ error between $\Pi^k_N u, \tilde{\Pi}^k_N u$ and $u \in W^{r,2}\LRp{D^k}$ is $O(h^r)$.  Theorem~\ref{thm:superconverge} demonstrates that the $L^2$ difference between $\Pi^k_N u, \tilde{\Pi}^k_N u$ is $O(h^{r+1})$, or at least one order higher than the approximation error.  We note that optimal convergence of the weight-adjusted projection requires that the geometric mapping $\bm{\Phi}_k$ is asymptotically affine (i.e., the Sobolev norm of $J^k, \bm{G}^k$ does not grow under mesh refinement), which is ensured under nested mesh refinement and appropriate curvilinear blending strategies \cite{lenoir1986optimal, warburton2013low, chan2016weight2}.  

\subsubsection{Local and global conservation}
\label{sec:conservation}
We next address local conservation of the weight-adjusted scheme (\ref{eq:dgform2}), which is also referred to as primary conservation \cite{fisher2013discretely, fisher2013high, carpenter2014entropy, friedrich2017entropy}.  We begin by noting that (\ref{eq:dgform2}) is locally conservative with respect to the weight-adjusted inner product.  Testing (\ref{eq:dgform2weak}) with $\bm{1}$ yields
\[
\bm{1}^T\bm{M}\bm{M}_{1/J^k}^{-1}\bm{M}\td{\bm{u}_h}{t} + 
\sum_{j=1}^d \bm{1}^T \LRp{2\bm{Q}^j_k \circ \bm{F}_{j,S}}\bm{1} + \sum_{j=1}^d \bm{1}^T\bm{W}_f \diag{\bm{n}_j}\LRp{\bm{f}_j^* - \bm{f}_j(\tilde{\bm{u}}_f)} = 0.
\]
Local conservation can be shown by applying the SBP property (\ref{eq:sbpk}) and noting that $\LRp{\bm{B}^i_k \circ \bm{F}_{j,S}}\bm{1} = \bm{W}_f\diag{\bm{n}_j \circ \bm{J}^k_f}\bm{f}_j(\bm{u})$ (using the consistency of $\bm{f}_S$ and diagonal nature of $\bm{B}^i_k$) yields
\[
\bm{1}^T\bm{M}\bm{M}_{1/J^k}^{-1}\bm{M}\td{\bm{u}_h}{t} + 
\sum_{j=1}^d \bm{1}^T \LRp{\LRp{\bm{Q}^j_k-\LRp{\bm{Q}^j_k}^T} \circ \bm{F}_{j,S}}\bm{1} + \sum_{j=1}^d \bm{1}^T\bm{W}_f \diag{\bm{n}_j}\LRp{\bm{f}_j^*} = 0.
\]
Since $\LRp{\bm{Q}^j_k-\LRp{\bm{Q}^j_k}^T} $ is skew-symmetric and $\bm{F}_{j,S}$ is symmetric, the Hadamard product of these two matrices is skew-symmetric.  As a result, $\bm{1}^T \LRp{\LRp{\bm{Q}^j_k-\LRp{\bm{Q}^j_k}^T} \circ \bm{F}_{j,S}}\bm{1} = 0$ and 
\begin{equation}
\bm{1}^T\bm{M}\bm{M}_{1/J^k}^{-1}\bm{M}\td{\bm{u}_h}{t} + \sum_{j=1}^d \bm{1}^T\bm{W}_f \diag{\bm{n}_j}\LRp{\bm{f}_j^*} = 0.
\label{eq:localcons}
\end{equation}
Global conservation is shown by summing (\ref{eq:localcons}) over all elements $D^k$.  Because the flux $\bm{f}_j^*$ is single-valued on each face and the outward normal $\bm{n}_j$ changes sign on adjacent elements, the contributions $\bm{1}^T\bm{W}_f \diag{\bm{n}_j}\LRp{\bm{f}_j^*}$ cancel on all interior interfaces.\footnote{Conservation still holds when incorporating entropy dissipation through penalty terms involving jumps.  This is because the definition of the jump changes sign on adjacent elements, such that jump contributions cancel when summing over all elements.} On periodic meshes, this yields conservation with respect to the weight-adjusted mass matrix
\[
\sum_k \bm{1}^T\bm{M}\bm{M}_{1/J^k}^{-1}\bm{M}\td{\bm{u}_h}{t} = 0.
\]
However, while the weight-adjusted approximation of the mass matrix is high order accurate and efficient, it does not preserve the average over a physical element, which is equivalent to the $J$-weighted average over the reference element.  This is due to the fact that, in general,
\begin{equation}
\int_{\hat{D}}u J^k \diff{\hat{\bm{x}}} - \int_{\hat{D}} T_{1/J^k}^{-1}u \diff{\hat{\bm{x}}} \approx \bm{1}^T\bm{M}_{J^k}\bm{u} - \bm{1}^T\bm{M}\bm{M}_{1/J^k}^{-1}\bm{M}\bm{u} \neq 0.  
\label{eq:conserr}
\end{equation}
Results in \cite{chan2016weight1} show that the difference between the true mean and weight-adjusted mean in (\ref{eq:conserr}) converges extremely fast at a rate of $O(h^{2N+2})$.  However, for systems of conservation laws, it is often desired that the local element average is preserved exactly up to machine precision.  We present two simpler approaches to ensuring local conservation in this section.  

The simplest way to ensure local conservation is to approximate $J$ using a degree $N$ polynomial and utilize a sufficiently accurate quadrature.  Then, we have the following lemma:
\begin{lemma}
Let $J^k \in P^N$, and let integrals be computed using quadrature which is exact for degree $2N$ polynomials.  Then, 
\[
\bm{1}^T\bm{M}\bm{M}^{-1}_{1/J^k}\bm{M}\bm{u} = \int_{\hat{D}}T^{-1}_{1/J^k} u \diff{\hat{\bm{x}}}=  \int_{\hat{D}} u J^k\diff{\hat{\bm{x}}}.
\]
\label{eq:conscorrect1}
\end{lemma}
\begin{proof}
The proof relies on (\ref{eq:wadgTw}), which states that $\LRp{\frac{1}{J^k} T^{-1}_{1/J^k}u,v}_{\hat{D}} = \LRp{u,v}_{\hat{D}}$ for all $v\in P^N$.  If $J^k$ is a polynomial of degree $N$, then taking $v = 1$ yields
\[
\LRp{T^{-1}_{1/J^k}u,1}_{\hat{D}} = \LRp{\frac{1}{J^k}T^{-1}_{1/J^k}u,{J^k}}_{\hat{D}}= \LRp{u,{J^k}}_{\hat{D}} = \int_{\hat{D}} uJ^k\diff{\hat{\bm{x}}}. 
\]
Additionally, the proof still holds if integrals are approximated using a quadrature rule which exactly integrates $uJ^k \in P^{2N}$.
\end{proof}
We note that, for isoparametric curved elements, $J^k \not\in P^N$ in general (in 2D, $J^k \in P^{2N-2}$, while in 3D, $J^k \in P^{3N-3}$ \cite{johnen2013geometrical}).  Thus, to ensure local conservation, we will approximate $J^k$ using a degree $N$ polynomial (for example, the interpolant or $L^2$ projection onto $P^N$).  We note that this approximation is required only in the weight-adjusted mass matrix, and does not modify the scaled geometric terms $G^k_{ij}$.  

The second approach relies on a simple correction which restores exact conservation of the true mean.  In \cite{chan2016weight2}, it was shown that a rank one correction of the weight-adjusted mass matrix inverse preserves the mean exactly.  However, this requires the use of the Sherman--Morrison formula to compute the inverse of a rank one matrix update, which can be cumbersome to incorporate.  We present a simpler explicit correction formula which does not involve matrices.  Let $u_J$ and $u_{\rm WADG}$ be defined through 
\begin{align*}
\LRp{u_J J,v}_{\hat{D}} &= \LRp{f,v}_{\hat{D}}, \qquad \forall v \in P^N\LRp{\hat{D}},\\
\LRp{T_{1/J}^{-1}u_{\rm WADG},v}_{\hat{D}} &= \LRp{f,v}_{\hat{D}}, \qquad \forall v \in P^N\LRp{\hat{D}},
\end{align*}
where $u_J$ corresponds to the inversion of the weighted mass matrix and $u_{\rm WADG}$ corresponds to the inversion of a weight-adjusted mass matrix.  For example, if $f = u J$ for some function $u(\bm{x})$, then $u_J = \Pi^k_N u$ and $u_{\rm WADG} = \tilde{\Pi}^k_N u$.  To ensure both local and global conservation, we require that the weighted average of $u_{\rm WADG}$ is the same as the weighted average of $u_J$.  Let the conservative weight-adjusted $\tilde{u}_{\rm WADG}$ be defined as
\begin{equation}
  \tilde{u}_{\rm WADG} = u_{\rm WADG} + \frac{\int_{\hat{D}} \left(f  - J u_{\rm WADG}\right)\diff{\hat{\bm{x}}}}{\int_{\hat{D}} J \diff{\hat{\bm{x}}}}.  
  \label{eq:conscorrect}
\end{equation}
Taking the weighted integral of $\tilde{u}_{\rm WADG}$ yields
\[
  \int_{\hat{D}}\tilde{u}_{\rm WADG}J \diff{\hat{\bm{x}}} = \int_{\hat{D}} f \diff{\hat{\bm{x}}}.
\]
In other words, (\ref{eq:conscorrect}) ensures local conservation by correcting the weighted average of $u_{\rm WADG}$ to match that of $u_J$.  Applying this correction to the right hand side of (\ref{eq:dgform2}) then yields a scheme which locally and globally conserves mean values of the conservative variables.

This approach is applicable to an arbitrary weight, and can be generalized to matrix-valued weights as well \cite{chan2017weight}.  Moreover, using Theorem 6 in \cite{chan2016weight1}, one can show that the $L^2$ norm of the difference $\tilde{u}_{\rm WADG} - u_{\rm WADG}$ is $O(h^{2N+1})$, and does not affect high order accuracy.  

\section{Enforcing the discrete geometric conservation law}
\label{sec:4}
An important aspect of Theorem~\ref{thm:stab1} is the assumption that $\bm{Q}^j_k\bm{1} = 0$ for $j = 1,\ldots,d$.  However, this is not always guaranteed to hold for $\bm{Q}^j_k$ as defined through (\ref{eq:dik}).  In this section, we discuss methods of constructing the geometric terms $G^k_{ij}$ for curvilinear meshes in a way that ensures $\bm{Q}^j_k\bm{1} = 0$.  

From (\ref{eq:dik}), the condition $\bm{Q}^j_k\bm{1} = 0$ is equivalent to
\begin{align*}
\bm{Q}^j_k\bm{1} = \bm{W}_N\sum_{j=1}^d\bm{D}^j_N\circ \avg{\bm{G}_{ij}^k}\bm{1} &= \frac{1}{2}\bm{W}_N\sum_{j=1}^d \LRp{ \diag{\bm{G}_{ij}^k}\bm{D}^j_N \bm{1} + \bm{D}^j_N\diag{\bm{G}_{ij}^k}\bm{1}} \\
&= \frac{1}{2}\bm{W}_N\sum_{j=1}^d \bm{D}^j_N \bm{G}_{ij}^k = 0,
\end{align*}
where we have used that $\bm{D}^j_N \bm{1} = 0$ to eliminate the first term.  Since $\bm{W}_N$ is a diagonal matrix with positive entries, $\bm{Q}^j_k\bm{1} = 0$ is equivalent to ensuring that a discrete version of the GCL (\ref{eq:gcl}) holds
\begin{equation}
\sum_{j=1}^d \bm{D}^j_N \bm{G}_{ij}^k = 0.
\label{eq:dgcl}
\end{equation}
This condition is required to ensure that free-stream preservation holds at the discrete level.  In other words, we wish to ensure that the semi-discrete scheme preserves (for $\bm{u}$ constant) 
\[
\pd{\bm{u}}{t} + \Grad \cdot \bm{f}(\bm{u}) = \pd{\bm{u}}{t} = 0.
\]
For isoparametric geometric mappings (where the degree of the mapping matches the degree of the polynomial approximation) in two dimensions, the GCL is naturally enforced by the ``cross-product'' form, noting that the scaled metric terms $G^k_{ij}$ are exactly polynomials of degree $N$.  As a result, computing the metric terms exactly automatically enforces that both the continuous GCL (\ref{eq:gcl}) and the discrete GCL (\ref{eq:dgcl}) are satisfied.  However, the discrete GCL is not always maintained at the discrete level in 3D.  

In three dimensions, geometric terms are typically computed in ``cross-product'' form
\begin{align}
\LRs{\begin{array}{c}
G^k_{1i}\\
G^k_{2i}\\
G^k_{3i}\end{array}} &= \pd{\bm{x}}{\hat{x}_j}\times \pd{\bm{x}}{\hat{x}_k}, \qquad (i,j,k) = (1,2,3), \text{ cyclic}.
\end{align}
Note the abuse of notation here and in the sequel, the superscript $k$ refers to the element number and the subscript $k$ to the cyclic index. 
This formula can be used to compute the geometric terms exactly at volume and surface quadrature points.  However, because $\pd{\bm{x}}{\hat{x}_j},\pd{\bm{x}}{\hat{x}_k} \in P^{N-1}$, the geometric terms $G^k_{ij}$ are polynomials of degree $P^{2N-2}$.  The discrete GCL condition holds only if $G^k_{ij} \in P^{2N-2}$ are differentiated exactly; however, because applying $\bm{D}^j_N$ involves the $L^2$ projection, and because $G^k_{ij}$ and its $L^2$ projection onto degree $N$ polynomials can differ, the discrete GCL (\ref{eq:dgcl}) does not hold in general \cite{kopriva2006metric}.  

This can be remedied by using an alternative form of the geometric terms, which ensures that (\ref{eq:dgcl}) is satisfied a-priori \cite{thomas1979geometric, visbal2002use, kopriva2006metric}.  The geometric terms $G^k_{ij}$ can also be computed using a ``conservative curl'' form, where $G^k_{ij}$ are the image of the curl applied to some quantity
 as follows: 
\begin{align}
\LRs{\begin{array}{c}
G^k_{1i}\\
G^k_{2i}\\
G^k_{3i}\end{array}} =
\LRs{\begin{array}{c}
    {\left(         - \hat{\nabla} \times \left(x_3 \hat{\nabla} x_2\right)\right)}_i\\
    {\left(\phantom{-}\hat{\nabla} \times \left(x_3 \hat{\nabla} x_1\right)\right)}_i\\
    {\left(\phantom{-}\hat{\nabla} \times \left(x_1 \hat{\nabla} x_2\right)\right)}_i\end{array}},
\label{eq:conscurl}
\end{align}
where the subscript $i$ denotes the $i$th component of the vector quantity.
From (\ref{eq:conscurl}), it can be observed that, because the divergence of a curl vanishes, the continuous GCL condition (\ref{eq:gcl}) holds.  
The central idea of \cite{visbal2002use, kopriva2006metric} is to use (\ref{eq:conscurl}), but to interpolate before applying the curl  
\begin{align}
\LRs{\begin{array}{c}
G^k_{1i}\\
G^k_{2i}\\
G^k_{3i}\end{array}} =
\LRs{\begin{array}{c}
    {\left(         - \hat{\nabla} \times I_N\left(x_3 \hat{\nabla} x_2\right)\right)}_i\\
    {\left(\phantom{-}\hat{\nabla} \times I_N\left(x_3 \hat{\nabla} x_1\right)\right)}_i\\
    {\left(\phantom{-}\hat{\nabla} \times I_N\left(x_1 \hat{\nabla} x_2\right)\right)}_i\end{array}},
\label{eq:iconscurl}
\end{align}
where $I_N$ denotes the degree $N$ polynomial interpolation operator. Since the geometric terms are still computed by applying a curl, the continuous GCL condition (\ref{eq:gcl}) is still satisfied. We shall also show that this approximation also satisfies the discrete GCL condition. 

We adopt a slight modification of (\ref{eq:iconscurl}) in this work which is tailored towards triangular and tetrahedral elements.   Because the geometric terms are computed by applying the curl, the geometric terms are approximated as degree $(N-1)$ polynomials rather than degree $N$ polynomials, which can reduce accuracy.   Instead, we approximate geometric terms by using the interpolation operator $I_{N+1}$ onto degree $(N+1)$ polynomials, then interpolating back to degree $N$ polynomials 
\begin{align}
\LRs{\begin{array}{c}
G^k_{1i}\\
G^k_{2i}\\
G^k_{3i}\end{array}} =
\LRs{\begin{array}{c}
    I_N{\left(         - \hat{\nabla} \times I_{N+1}\left(x_3 \hat{\nabla} x_2\right)\right)}_i\\
    I_N{\left(\phantom{-}\hat{\nabla} \times I_{N+1}\left(x_3 \hat{\nabla} x_1\right)\right)}_i\\
    I_N{\left(\phantom{-}\hat{\nabla} \times I_{N+1}\left(x_1 \hat{\nabla} x_2\right)\right)}_i\end{array}}.
\label{eq:iconscurl2}
\end{align}
For any $u \in P^N\LRp{\hat{D}}$, $\pd{u}{\hat{x}_i} \in P^{N-1}\LRp{\hat{D}}$ for $i = 1,2,3$.  Thus, the interpolation to degree $N$ polynomials is exact, since the derivatives of a degree $(N+1)$ polynomial are degree $N$ on triangles and tetrahedra.  

\begin{remark}
The accuracy of (\ref{eq:iconscurl2}) depends on the choice of interpolation points.  It is well known that interpolating at equispaced points can result in inaccurate polynomial approximations.  One can determine good interpolation point sets by optimizing over some measure of interpolation quality (such as the Lebesgue constant), and in practice, sets of interpolation points are pre-computed for some polynomial degrees $N = 1,\ldots, N_{\max}$ on the reference element and stored \cite{chen1996optimal, hesthaven1998electrostatics}, or explicitly computed as the image of equispaced points under an appropriately defined mapping \cite{blyth2006lobatto, warburton2006explicit, chan2015comparison}.  
\end{remark}

To prove that the construction (\ref{eq:iconscurl2}) satisfies Assumption~\ref{ass:norm}, we must assume that the interpolation points for a degree $N$ element include an appropriate number of points on each face.  We note that these assumptions exclude interpolation points which lie purely in the interior of an element, such as those introduced in \cite{williams2014symmetric, witherden2015identification}.  We can now show that the geometric terms satisfy all conditions necessary to guarantee entropy stability:
\begin{theorem}
Let the mesh consist of triangles or tetrahedral elements which satisfy Assumption~\ref{ass:norm}, and let the interpolation points which define the degree $N$ interpolation operator be distributed such that $N^f_p$ points lie on each face. Then, the approximate geometric terms $\tilde{J^k\bm{G}^k_{ij}}$  and approximate scaled normals $\tilde{n_iJ^k_f}$ computed using (\ref{eq:iconscurl2}) and (\ref{eq:normalconsistency})) satisfy both the discrete GCL condition (\ref{eq:dgcl}) and Assumption~\ref{ass:norm}.  Additionally, the error in the approximation satisfies
\[
\nor{{G^k_{ij}}-\tilde{G^k_{ij}}}_{L^2\LRp{\Omega}}^2  \leq C_N \LRb{\Omega} h^{N+2} \sqrt{\sum_{i=1}^d \sum_k \nor{\bm{r}_i}^2_{W^{N+2,2}\LRp{D^k}}}.
\]
\label{thm:koprivagcl}
\end{theorem}
\begin{proof}
  The satisfaction of (\ref{eq:dgcl}) relies on the fact that $\tilde{G^k_{ij}}$ is a degree $N$ polynomial and is equal to its own $L^2$ projection.  Let $\tilde{\bm{G}^k_{ij}}$ denote the polynomial coefficients of $\tilde{G^k_{ij}}$.  
Then, applying $\bm{D}^j_N$ to evaluations of $\tilde{G^k_{ij}}$ at volume and surface quadrature points and using (\ref{eq:dnvqvf}), we have
\[
\sum_{j=1}^d\bm{D}^j_N \LRs{
\begin{array}{c}
\bm{V}_q\\
\bm{V}_f
\end{array}}
\tilde{\bm{G}^k_{ij}} = 
\sum_{j=1}^d \LRs{\begin{array}{c}
\bm{V}_q\bm{D}_i\tilde{\bm{G}^k_{ij}} \\
0
\end{array}}.
\]
The entries of $\bm{V}_q\bm{D}_i\tilde{\bm{G}^k_{ij}}$ correspond to values of the derivatives of $\tilde{G^k_{ij}}$ evaluated at quadrature points.  Since 
\[
\sum_{j=1}^d\pd{}{\hat{x}_j}\tilde{G^k_{ij}} = 0
\]
by construction using (\ref{eq:iconscurl2}), $\sum_{j=1,\ldots,d} \bm{V}_q\bm{D}_i\tilde{\bm{G}^k_{ij}} = \bm{0}$ as well.  

Equation (\ref{eq:normalconsistency}) of Assumption~\ref{ass:norm} is satisfied by directly constructing the scaled normals $\tilde{\bm{n}}J^k_f$ using values of $\tilde{G^k_{ij}}$ at quadrature points.  We must now prove that the construction of $\tilde{\bm{n}}J^k_f$ implies that Equation (\ref{eq:normalsign}) holds. This is not immediately clear; since the normals are constructed from the approximate geometric terms and the formula (\ref{eq:normalconsistency}), it is not guaranteed that $\tilde{\bm{n}}^+J^{k,+}_f = -\tilde{\bm{n}}J^k_f$ will hold across a shared face.  However, the scaled normal vectors involve only nodal values on the shared face because the normals are computed in terms of the tangential reference derivatives \cite{kopriva2006metric}.  Thus, assuming a watertight mesh, the interpolation nodes on two neighboring elements will coincide for a shared face $f$, such that the trace of $\tilde{G^k_{ij}}$ from either neighboring element will be the same lower-dimensional polynomial on $f$.  This is sufficient to ensure that the scaled normal vectors $\tilde{\bm{n}}^+J^{k,+}_f, \tilde{\bm{n}}J^k_f$ will be equal and opposite.  

The local $L^2$ error $\nor{{G^k_{ij}}-\tilde{G^k_{ij}}}_{L^2\LRp{D^k}}$ can be bounded by noting that, since the error $G^k_{ij}-\tilde{G^k_{ij}}$ consists of linear combinations of derivatives of the interpolation error $\bm{r}_i-I_{N+1}\bm{r}_i$, it can be bounded by the $H^1$-seminorm of the latter quantity
\begin{align*}
\nor{{G^k_{ij}}-\tilde{G^k_{ij}}}_{L^2\LRp{D^k}} &\leq 
 \sum_{i=1}^d C_1 \nor{\sqrt{J^k}}_{L^2\LRp{\hat{D}}} \LRb{\LRp{\bm{r}_i-I_{N+1}\bm{r}_i}}_{H^1\LRp{\hat{D}}}\\
 &\leq  \sum_{i=1}^d \tilde{C}_N \nor{\sqrt{J^k}}_{L^2\LRp{\hat{D}}} \LRb{\bm{r}_i}_{W^{N+2,2}\LRp{\hat{D}}},
\end{align*}
where we have used the Bramble-Hilbert lemma \cite{brenner2007mathematical} on the reference element in the last step.  Since it is applied on the reference element $\hat{D}$ rather than the physical element $D^k$, the constant $\tilde{C}_N$ depends on the reference element and order of approximation, but not the mesh size $h$.  A scaling argument for quasi-uniform meshes then yields that
\begin{align}
\LRb{\bm{r}_i}_{W^{N+2,2}\LRp{\hat{D}}} \leq C_2 h^{N+2}\nor{\bm{r}_i}_{W^{N+2,2}\LRp{D^k}}.
\label{eq:localgeoerr}
\end{align}
The global estimate results from squaring (\ref{eq:localgeoerr}), summing over all elements and using $\nor{\sqrt{J^k}}^2_{L^2\LRp{\hat{D}}} = \LRb{D^k}$.
\end{proof}

\begin{remark}
It should be pointed out that this approach does not work on hexahedral elements.  This is due to the fact that the natural polynomial space on hexahedral elements is the tensor product space $Q^N\LRp{\widehat{D}}$
\[
Q^N\LRp{\widehat{D}} = \LRc{\hat{x}_1^{i_1}\ldots\hat{x}_d^{i_d}, \quad \hat{\bm{x}} \in \widehat{D}, \quad 0 \leq i_k \leq N}.
\]
For $u\in Q^{N+1}\LRp{\hat{D}}$, $\pd{u}{\hat{x}_i}$ is at most degree $N$ in the coordinate $\hat{x}_i$, but can remain a polynomial of degree $N+1$ in all other coordinates.  Thus, interpolating from degree $N+1$ to degree $N$ polynomials in (\ref{eq:iconscurl2}) introduces aliasing errors and is no longer exact.  The result of (\ref{eq:iconscurl2}) is no longer the image of a curl, and thus does not satisfy the discrete GCL by construction.  
\end{remark}

We briefly outline how to compute $\tilde{G^k_{ij}}$ in three dimensions.  Let $\LRc{\hat{\bm{x}}^N_i}_{i=1}^{N_p}$ denote the set of degree $N$ interpolation points, and let $\ell^N_i(\hat{\bm{x}})$ denote the $i$th degree $N$ Lagrange basis function on the reference element.  We define interpolation matrices $\bm{V}_N^{N+1}$ and $\bm{V}^N_{N+1}$ between degree $N$ and $N+1$ polynomials such that 
\begin{align}
\LRp{\bm{V}_N^{N+1}}_{ij} &= \ell^N_j(\hat{\bm{x}}^{N+1}_i) , \qquad 1\leq i \leq N_p, \qquad \qquad 1\leq i \leq (N+1)_p\\
\LRp{\bm{V}^N_{N+1}}_{ij} &= \ell^{N+1}_j(\hat{\bm{x}}^{N}_i) , \qquad 1\leq i \leq (N+1)_p, \qquad \qquad 1\leq i \leq N_p, \nonumber
\end{align}
where $N_p, (N+1)_p$ denotes the number of interpolation points for degree $N$ and $N+1$ polynomials, respectively.  Next, let $\bm{x}_1,\bm{x}_2,\bm{x}_3$ denote vectors containing $x_1,x_2,x_3$ coordinates of degree $N$ interpolation points on a curved physical element $D^k$, and let $\tilde{\bm{x}}_1,\tilde{\bm{x}}_2,\tilde{\bm{x}}_3$ denote their evaluation at degree $(N+1)$ interpolation points
\begin{align}
\tilde{\bm{x}}_1 = \bm{V}_N^{N+1}\bm{x}_1, \qquad \tilde{\bm{x}}_2 = \bm{V}_N^{N+1}\bm{x}_2, \qquad \tilde{\bm{x}}_3 = \bm{V}_N^{N+1}\bm{x}_3.
\end{align}
Let $\tilde{\bm{D}}^{N+1}_i$ denote the nodal differentiation matrix of degree $N+1$ with respect to the $i$th coordinate direction.  The geometric factors are computed as follows:
\begin{align}
\bm{G^k}_{11} &= \phantom{-}\bm{V}_{N+1}^N\LRp{\tilde{\bm{D}}^{N+1}_3 \LRp{\LRp{ \tilde{\bm{D}}^{N+1}_2 \tilde{\bm{x}}_2} \circ\tilde{\bm{x}}_3} - \tilde{\bm{D}}^{N+1}_2 \LRp{ \LRp{\tilde{\bm{D}}^{N+1}_3\tilde{\bm{x}}_2} \circ\tilde{\bm{x}}_3}} \\
\bm{G^k}_{12} &= \phantom{-}\bm{V}_{N+1}^N\LRp{\tilde{\bm{D}}^{N+1}_1 \LRp{ \LRp{\tilde{\bm{D}}^{N+1}_3\tilde{\bm{x}}_2} \circ\tilde{\bm{x}}_3} - \tilde{\bm{D}}^{N+1}_3 \LRp{ \LRp{\tilde{\bm{D}}^{N+1}_1\tilde{\bm{x}}_2 }\circ\tilde{\bm{x}}_3}} \nonumber\\
\bm{G^k}_{13} &= \phantom{-}\bm{V}_{N+1}^N{\LRp{\tilde{\bm{D}}^{N+1}_2 \LRp{\LRp{ \tilde{\bm{D}}^{N+1}_1\tilde{\bm{x}}_2} \circ\tilde{\bm{x}}_3} - \tilde{\bm{D}}^{N+1}_1 \LRp{\LRp{ \tilde{\bm{D}}^{N+1}_2\tilde{\bm{x}}_2} \circ\tilde{\bm{x}}_3}}} \nonumber\\
\bm{G^k}_{21} &= -\bm{V}_{N+1}^N\LRp{\tilde{\bm{D}}^{N+1}_3 \LRp{\LRp{ \tilde{\bm{D}}^{N+1}_2 \tilde{\bm{x}}_1} \circ\tilde{\bm{x}}_3} - \tilde{\bm{D}}^{N+1}_2 \LRp{\LRp{ \tilde{\bm{D}}^{N+1}_3\tilde{\bm{x}}_1} \circ\tilde{\bm{x}}_3}} \nonumber \\
\bm{G^k}_{22} &= -\bm{V}_{N+1}^N\LRp{\tilde{\bm{D}}^{N+1}_1 \LRp{ \LRp{\tilde{\bm{D}}^{N+1}_3\tilde{\bm{x}}_1} \circ\tilde{\bm{x}}_3} - \tilde{\bm{D}}^{N+1}_3 \LRp{ \LRp{\tilde{\bm{D}}^{N+1}_1\tilde{\bm{x}}_1 }\circ\tilde{\bm{x}}_3}} \nonumber\\
\bm{G^k}_{23} &= -\bm{V}_{N+1}^N\LRp{\tilde{\bm{D}}^{N+1}_2 \LRp{ \LRp{\tilde{\bm{D}}^{N+1}_1\tilde{\bm{x}}_1} \circ\tilde{\bm{x}}_3} - \tilde{\bm{D}}^{N+1}_1 \LRp{ \LRp{\tilde{\bm{D}}^{N+1}_2\tilde{\bm{x}}_1 }\circ\tilde{\bm{x}}_3}} \nonumber\\
\bm{G^k}_{31} &= -\bm{V}_{N+1}^N\LRp{\tilde{\bm{D}}^{N+1}_3 \LRp{ \LRp{\tilde{\bm{D}}^{N+1}_2 \tilde{\bm{x}}_2} \circ\tilde{\bm{x}}_1} - \tilde{\bm{D}}^{N+1}_2 \LRp{ \LRp{\tilde{\bm{D}}^{N+1}_3\tilde{\bm{x}}_2} \circ\tilde{\bm{x}}_1}} \nonumber\\
\bm{G^k}_{32} &= -\bm{V}_{N+1}^N\LRp{\tilde{\bm{D}}^{N+1}_1 \LRp{ \LRp{\tilde{\bm{D}}^{N+1}_3\tilde{\bm{x}}_2} \circ\tilde{\bm{x}}_1} - \tilde{\bm{D}}^{N+1}_3 \LRp{  \LRp{\tilde{\bm{D}}^{N+1}_1\tilde{\bm{x}}_2 }\circ\tilde{\bm{x}}_1}} \nonumber\\
\bm{G^k}_{33} &= -\bm{V}_{N+1}^N\LRp{\tilde{\bm{D}}^{N+1}_2 \LRp{ \LRp{\tilde{\bm{D}}^{N+1}_1\tilde{\bm{x}}_2} \circ\tilde{\bm{x}}_1} - \tilde{\bm{D}}^{N+1}_1 \LRp{  \LRp{\tilde{\bm{D}}^{N+1}_2\tilde{\bm{x}}_2 }\circ\tilde{\bm{x}}_1}}. \nonumber
\end{align}
\begin{remark}

We note that the discrete GCL (\ref{eq:dgcl}) can also be enforced directly through a local constrained minimization problem \cite{fernandez2016simultaneous, crean2018entropy}, which yields a solution in terms of a pseudo-inverse.  However, we have not found a straightforward way to simultaneously enforce both the discrete GCL condition (\ref{eq:dgcl}) and Assumption~\ref{ass:norm} using this approach.
\end{remark}

\section{Numerical experiments}
\label{sec:num}

In this section, we present numerical results which verify the theoretical results in this work.  We first verify that the weight-adjusted $L^2$ projection $\tilde{\Pi}^k_N$ and the GCL-satisfying geometric factors $G^k_{ij}$ obey the error estimates in Theorem~\ref{thm:superconverge} and Theorem~\ref{thm:koprivagcl}.  Next, we verify the semi-discrete entropy conservation, primary conservation, and accuracy of the proposed high order accurate methods for the compressible Euler equations on curved meshes in two and three dimensions.  For all curved meshes, we utilize the low storage weight-adjusted formulation (\ref{eq:dgform2}).  

In choosing the timestep $dt$, we follow \cite{chan2015gpu} and set 
\begin{equation}
dt = \min_{k} {\rm CFL} \frac{\nor{J}_{L^{\infty}\LRp{D^k}}}{C_N\nor{J^f}_{L^{\infty}\LRp{\partial D^k}}},
\end{equation}
where $C_N$ is the $O(N^2)$ order-dependent constant in the surface trace inequality \cite{chan2015gpu} and ${\rm CFL}$ is a user-chosen constant.  All numerical experiments in this work utilize the five-stage fourth order low storage Runge--Kutta (LSRK-45) time-stepper \cite{carpenter1994fourth}.

\subsection{Accuracy of weight-adjusted projection and geometric terms}

In this section, we verify Theorems~\ref{thm:superconverge} and \ref{thm:koprivagcl} concerning the accuracy of the weight-adjusted projection and modified construction of geometric terms satisfying the discrete geometric conservation law.  Figure~\ref{fig:superconverge} shows $L^2$ errors for both the standard $L^2$ projection (\ref{eq:l2curv}) and the weight-adjusted projection (\ref{eq:wadgproj}) for a series of warped meshes of degree $N = 4$.  Errors are estimated using degree $2N+1$ quadratures for triangles and tetrahedra \cite{xiao2010quadrature}.  We compute $L^2$ errors for both smooth and discontinuous functions
\[
f(\bm{x}) = e^{x_1+x_2}\sin\LRp{{\pi x_1}}\sin\LRp{{\pi x_2}}, \qquad g(\bm{x}) = f(\bm{x}) + H(x_1+x_2-\sin(\pi x_1)),
\]
where $H$ is the Heaviside function.  For the smooth function $f(\bm{x})$, we observe that the difference between the $L^2$ and weight-adjusted projections indeed converges at a rate of $O\LRp{h^{N+2}}$ as predicted by Theorem~\ref{fig:superconverge}, such that the $L^2$ errors for each projection appear identical.  The $L^2$ errors for the $L^2$ and weight-adjusted projections of the discontinuous function $g(\bm{x})$ are also virtually identical.  However, the difference beteween the $L^2$ and weight-adjusted projection converges faster than estimated by Theorem~\ref{fig:superconverge}, with the $L^2$ error converging as $O(h^{1/2})$ and the difference converging as $O(h^{2+1/2})$. 


\begin{figure}
\centering
\begingroup
\captionsetup[subfigure]{width=.275\textwidth}
\subfloat[Warped curvilinear mesh]{\raisebox{1.75em}{\includegraphics[width=.25\textwidth]{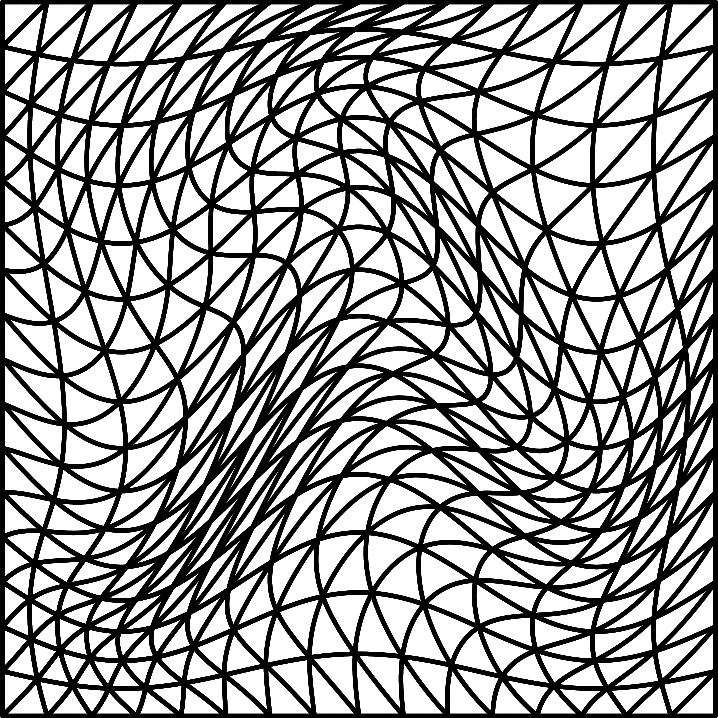}}}
\subfloat[Smooth function]{
\begin{tikzpicture}
\begin{loglogaxis}[
    legend cell align=left,
    legend style={legend pos=south east, font=\tiny},
    width=.37\textwidth,    
    xlabel={Mesh size $h$},
    ylabel={$L^2$ error}, 
     ymin=1e-11, ymax=1e-1,    
    grid style=dashed,
] 
\addplot[color=blue,mark=*,semithick, mark options={solid,fill=markercolor}]
coordinates{(0.5,0.0268241)(0.25,0.00198388)(0.125,7.29929e-05)(0.0625,2.43094e-06)(0.03125,7.71592e-08)};
\logLogSlopeTriangleFlip{0.3}{0.15}{0.48}{5}{blue}
\addplot[color=red,mark=x,dashed,semithick, mark options={solid,fill=markercolor}]
coordinates{(0.5,0.0269391)(0.25,0.00197894)(0.125,7.29563e-05)(0.0625,2.43063e-06)(0.03125,7.71567e-08)};
\addplot[color=black,mark=triangle*,semithick, mark options={solid,fill=markercolor}]
coordinates{(0.5,0.00237193)(0.25,5.41179e-05)(0.125,7.6934e-07)(0.0625,1.18362e-08)(0.03125,1.84801e-10)};
\logLogSlopeTriangleFlip{0.3}{0.15}{0.225}{6}{black}

\legend{$L^2$ projection,Weight-adjusted,Difference}
\end{loglogaxis}
\end{tikzpicture}
}
\subfloat[Discontinuous function]{
\begin{tikzpicture}
\begin{loglogaxis}[
    legend cell align=left,
    legend style={legend pos=south east, font=\tiny},
    width=.37\textwidth,
    xlabel={Mesh size $h$},
    grid style=dashed,
] 

\addplot[color=blue,mark=*,semithick, mark options={solid,fill=markercolor}]
coordinates{(0.5,0.323258)(0.25,0.227822)(0.125,0.160387)(0.0625,0.114753)(0.03125,0.0792918)};
\logLogSlopeTriangleFlip{0.4}{0.25}{0.875}{.5}{blue}
\addplot[color=red,mark=x,dashed,semithick, mark options={solid,fill=markercolor}]
coordinates{(0.5,0.324117)(0.25,0.227927)(0.125,0.160411)(0.0625,0.114756)(0.03125,0.0792924)};
\addplot[color=black,mark=triangle*,semithick, mark options={solid,fill=markercolor}]
coordinates{(0.5,0.00466574)(0.25,0.000804201)(0.125,0.000182886)(0.0625,2.96701e-05)(0.03125,5.01885e-06)};
\logLogSlopeTriangleFlip{0.3}{0.15}{0.18}{2.5}{black}

\end{loglogaxis}
\end{tikzpicture}
}
\endgroup
\caption{$L^2$ errors in approximating both smooth and discontinuous functions using $L^2$ and weight-adjusted projections on a curved mesh.  The approximation order is $N=4$, and a degree $2N$ quadrature rule is used to compute integrals over the reference triangle. }
\label{fig:superconverge}
\end{figure}
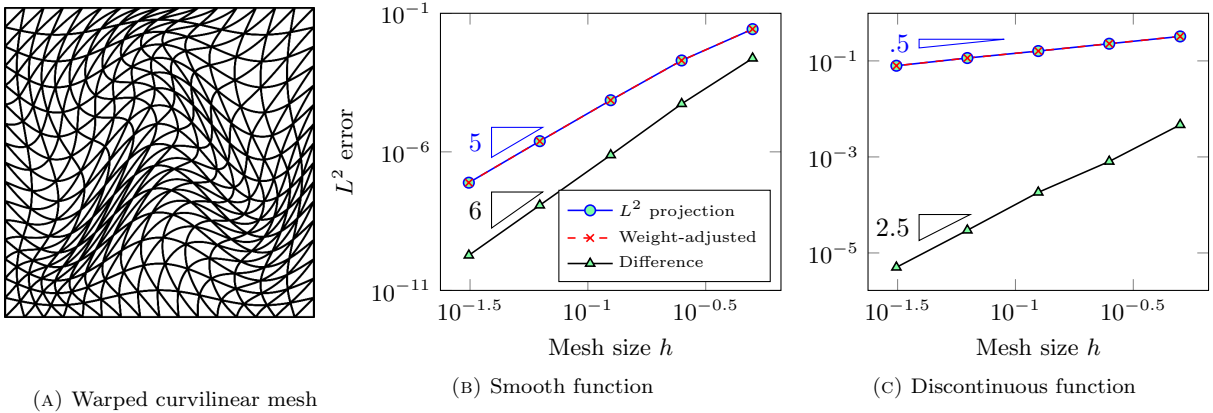

We next compare the approximation of the geometric factors on a curved three-dimensional mesh.  We generate a sequence of quasi-uniform unstructured tetrahedral meshes using GMSH \cite{geuzaine2009gmsh} and construct a curvilinear mesh from the distorted coordinates $\tilde{\bm{x}} = \bm{x} + \frac{1}{8}\cos\LRp{\frac{\pi}{2}x_1}\cos\LRp{\frac{\pi}{2}x_2}\cos\LRp{\frac{\pi}{2}x_3}$.
We then compute the $L^2$ error in approximating geometric terms for each element $D^k$ by computing $G^k_{ij}-\tilde{G^k_{ij}}$ at quadrature points.  
We estimate the mesh size as $h = \max_k \nor{J^k/J^k_f}_{L^{\infty}}$, since $J^k = O(h^d)$ and $J^k_f = O(h^{d-1})$ in $d$ dimensions \cite{chan2015gpu}.  Figure~\ref{fig:geomerr} shows errors for an $N=3$ and $N=4$ mesh.  We refer to the construction of approximate geometric terms $\tilde{G^k_{ij}}$ introduced in \cite{kopriva2006metric, hindenlang2012explicit} as ``Geo-$N$'', since the interpolation is performed using a degree $N$ interpolation operator.  We refer to the construction of $\tilde{G^k_{ij}}$ in (\ref{eq:iconscurl2}) and Theorem~\ref{thm:koprivagcl} as ``Geo-$(N+1)$'', since the main interpolation step is performed on degree $(N+1)$ polynomials instead.  It can be observed that the Geo-N scheme converges at a rate of $O(h^{N+1})$, while the Geo-(N+1) scheme converges at a rate of $O(h^{N+2})$.  We note that the error in both the Geo-$N$ and Geo-$(N+1)$ approximations of $\tilde{G^k_{ij}}$ converge at the same rate or faster than the best approximation error.  

\begin{figure}
\centering
\subfloat[$N=3$]{
\begin{tikzpicture}
\begin{loglogaxis}[
    legend cell align=left,
    legend style={legend pos=south east, font=\tiny},
    width=.45\textwidth,    
    xlabel={Mesh size $h$},
    ylabel={$L^2$ error}, 
    grid style=dashed,
] 

\addplot[color=blue,mark=*,semithick, mark options={solid,fill=markercolor}]
coordinates{(0.67574,0.00231891)(0.374388,7.21634e-05)(0.16654,3.2502e-06)(0.0867431,1.31183e-07)};
\addplot[color=magenta,dashed,semithick, mark options={solid,fill=markercolor}]
coordinates{(0.67574,0.00231891)(0.374388,0.000121059)(0.16654,2.10855e-06)(0.0867431,8.08278e-08)};
\addplot[color=red,mark=square*,semithick, mark options={solid,fill=markercolor}]
coordinates{(0.67574,0.00635039)(0.374388,0.000542466)(0.16654,4.85003e-05)(0.0867431,4.03792e-06)};
\addplot[color=black,dashed,semithick, mark options={solid,fill=markercolor}]
coordinates{(0.67574,0.00952559)(0.374388,0.000897556)(0.16654,3.51441e-05)(0.0867431,2.58651e-06)};


\legend{Geo-$(N+1)$, $h^{N+2}$, Geo-$N$, $h^{N+1}$}
\end{loglogaxis}
\end{tikzpicture}
}
\subfloat[$N=4$]{
\begin{tikzpicture}
\begin{loglogaxis}[
    legend cell align=left,
    legend style={legend pos=south east, font=\tiny},
    width=.45\textwidth,
    xlabel={Mesh size $h$},
    grid style=dashed,
] 

\addplot[color=blue,mark=*,semithick, mark options={solid,fill=markercolor}]
coordinates{(0.673351,0.000766195)(0.373777,1.44835e-05)(0.166655,2.91207e-07)(0.0867501,5.69685e-09)};
\addplot[color=magenta,dashed,semithick, mark options={solid,fill=markercolor}]
coordinates{(0.673351,0.000957743)(0.373777,2.80204e-05)(0.166655,2.20143e-07)(0.0867501,4.37946e-09)};
\addplot[color=red,mark=square*,semithick, mark options={solid,fill=markercolor}]
coordinates{(0.673351,0.00231677)(0.373777,0.000127702)(0.166655,5.3742e-06)(0.0867501,2.11781e-07)};
\addplot[color=black,dashed,semithick, mark options={solid,fill=markercolor}]
coordinates{(0.673351,0.00405434)(0.373777,0.000213685)(0.166655,3.7653e-06)(0.0867501,1.43901e-07)};


\legend{Geo-$(N+1)$, $h^{N+2}$, Geo-$N$, $h^{N+1}$}
\end{loglogaxis}
\end{tikzpicture}
}
\caption{$L^2$ errors in the approximation of metric terms $G^k_{ij}$ with metric terms $\tilde{G^k_{ij}}$ satisfying the discrete GCL condition (\ref{eq:dgcl}) and Assumption~\ref{ass:norm}.}
\label{fig:geomerr}
\end{figure}
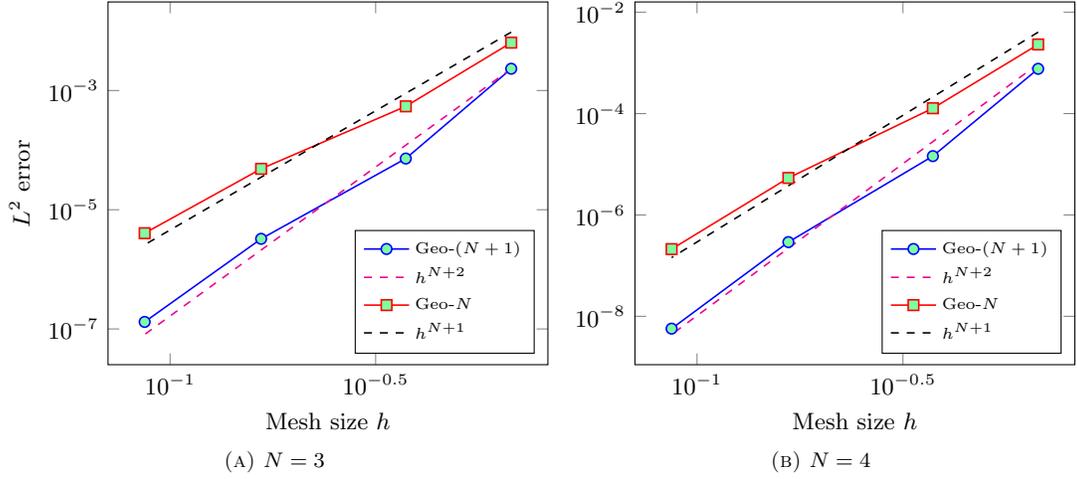

\subsection{Two-dimensional compressible Euler equations}

The compressible Euler equations in two dimensions are given as follows:
\begin{align}
\pd{\rho}{t} + \pd{\LRp{\rho u}}{x_1} + \pd{\LRp{\rho v}}{x_2} &= 0,\\
\pd{\rho u}{t} + \pd{\LRp{\rho u^2 + p }}{x_1} + \pd{\LRp{\rho uv}}{x_2} &= 0,\nonumber\\
\pd{\rho v}{t} + \pd{\LRp{\rho uv}}{x_1} + \pd{\LRp{\rho v^2 + p }}{x_2} &= 0,\nonumber\\
\pd{E}{t} + \pd{\LRp{u(E+p)}}{x_1} + \pd{\LRp{v(E+p)}}{x_2}&= 0.\nonumber
\end{align}
In two dimensions, the pressure is $p = (\gamma-1)\LRp{E - \frac{1}{2}\rho (u^2+v^2)}$, and the specific internal energy is $\rho e = E - \frac{1}{2}\rho (u^2+v^2)$.  

The choice of convex entropy for the Euler equations is non-unique \cite{harten1983symmetric}.  However, a unique entropy can be chosen by restricting to choices of entropy variables which symmetrize the viscous heat conduction term in the compressible Navier-Stokes equations \cite{hughes1986new}.  This leads to $U(\bm{u})$ of the form
\begin{equation}
U(\bm{u}) = -\frac{\rho s}{\gamma-1},
\label{eq:entropy2d}
\end{equation}
where $s = \log\LRp{\frac{p}{\rho^\gamma}}$ is the physical specific entropy. The entropy variables in two dimensions are 
\begin{align}
v_1 = \frac{\rho e (\gamma + 1 - s) - E}{\rho e}, \qquad v_2 = \frac{\rho u}{\rho e}, \qquad v_3 = \frac{\rho v}{\rho e}, \qquad v_4 = -\frac{\rho}{\rho e}.
\end{align}
The conservation variables in terms of the entropy variables are given by
\begin{equation}
\rho = -(\rho e) v_4, \qquad \rho u = (\rho e) v_2, \qquad \rho v = (\rho e) v_3, \qquad E = (\rho e)\LRp{1 - \frac{{v_2^2+v_3^2}}{2 v_4}},
\end{equation}
where $\rho e$ and $s$ in terms of the entropy variables are 
\begin{equation}
\rho e = \LRp{\frac{(\gamma-1)}{\LRp{-v_4}^{\gamma}}}^{1/(\gamma-1)}e^{\frac{-s}{\gamma-1}}, \qquad s = \gamma - v_1 + \frac{{v_2^2+v_3^2}}{2v_4}.
\end{equation}
The entropy conservative numerical fluxes for the two-dimensional compressible Euler equations are given by Chandrashekar \cite{chandrashekar2013kinetic}
\begin{align}
&f^1_{1,S}(\bm{u}_L,\bm{u}_R) = \avg{\rho}^{\log} \avg{u},& &f^1_{2,S}(\bm{u}_L,\bm{u}_R) = \avg{\rho}^{\log} \avg{v},&\\
&f^2_{1,S}(\bm{u}_L,\bm{u}_R) = f^1_{1,S} \avg{u} + p_{\rm avg},&  &f^2_{2,S}(\bm{u}_L,\bm{u}_R) = f^1_{2,S} \avg{u},&\nonumber\\
&f^3_{1,S}(\bm{u}_L,\bm{u}_R) = f^2_{2,S},& &f^3_{2,S}(\bm{u}_L,\bm{u}_R) = f^1_{2,S} \avg{v} + p_{\rm avg},&\nonumber\\
&f^4_{1,S}(\bm{u}_L,\bm{u}_R) = \LRp{E_{\rm avg} + p_{\rm avg}}\avg{u},& &f^4_{2,S}(\bm{u}_L,\bm{u}_R) = \LRp{E_{\rm avg} + p_{\rm avg} }\avg{v},& \nonumber
\end{align}
where we have defined the auxiliary quantities 
\begin{gather}
p_{\rm avg} = \frac{\avg{\rho}}{2\avg{\beta}}, \qquad E_{\rm avg} = \frac{\avg{\rho}^{\log}}{2\avg{\beta}^{\log}\LRp{\gamma -1}}   + \frac{\nor{\bm{u}}^2_{\rm avg}}{2}, \\
 \nor{\bm{u}}^2_{\rm avg} = 2(\avg{u}^2 + \avg{v}^2) - \LRp{\avg{u^2} +\avg{v^2}} \nonumber.  
\end{gather}

\subsubsection{Entropy conservation}

We begin by testing the propagation of a shock on a two-dimensional curved mesh using a discontinuous profile on the domain $\Omega = [0,20] \times [-5,5]$.   We set the initial velocities to be zero, and initialize the density and pressure as a discontinuous square pulse as in \cite{chan2017discretely}
\begin{equation}
\rho(\bm{x},t) = \begin{cases}
3 & \LRb{x_1} < 1/2 \text{ and } \LRb{x_2} < 1/2\\
2 & \text{otherwise},
\end{cases} \qquad 
u(\bm{x},t) = v(\bm{x},t) = 0, \qquad
p(\bm{x},t) = \rho^\gamma.
\label{eq:discontin}
\end{equation}
To test the scheme (\ref{eq:dgform2}), we utilize an entropy conservative flux and run it on a uniform triangular mesh with a curvilinear warping shown in Figure~\ref{fig:warp2d}.   Theorem~\ref{thm:stab2} ensures that, under an entropy conservative flux, (\ref{eq:dgform2}) is semi-discretely entropy conservative.  This does not hold at the fully discrete level; however, it is possible to verify that (\ref{eq:dgform2}) is entropy stable using other approaches.  

First, can examine the entropy RHS, which we define as the right hand side of (\ref{eq:dgform2}) tested with $\bm{v}_h$ 
\begin{equation}
  \text{entropy RHS}= -\sum_{j=1}^d\LRp{ \tilde{\bm{v}}^T\LRp{2\bm{Q}^j_k \circ \bm{F}_{j,S}}\bm{1} + \tilde{\bm{v}}_f^T\bm{W}_f^k \diag{\bm{n}_j}\LRp{\bm{f}_j^* - \bm{f}_j(\tilde{\bm{u}}_f)}}.  
\label{eq:entropyrhs}
\end{equation}
For positive density and pressure, (\ref{eq:entropyrhs}) should be zero to machine precision.  We can also track the change in entropy $\Delta U = \LRb{U(\bm{u}(\bm{x},t))-U(\bm{u}(\bm{x},0))}$, which should converge to $0$ as the timestep $dt$ approaches zero.  Furthermore, the rate of convergence should match the order of the time-stepper used \cite{gassner2016well, chan2017discretely}.  

\begin{figure}
\centering
\subfloat[Uniform mesh]{\includegraphics[width=.35\textwidth]{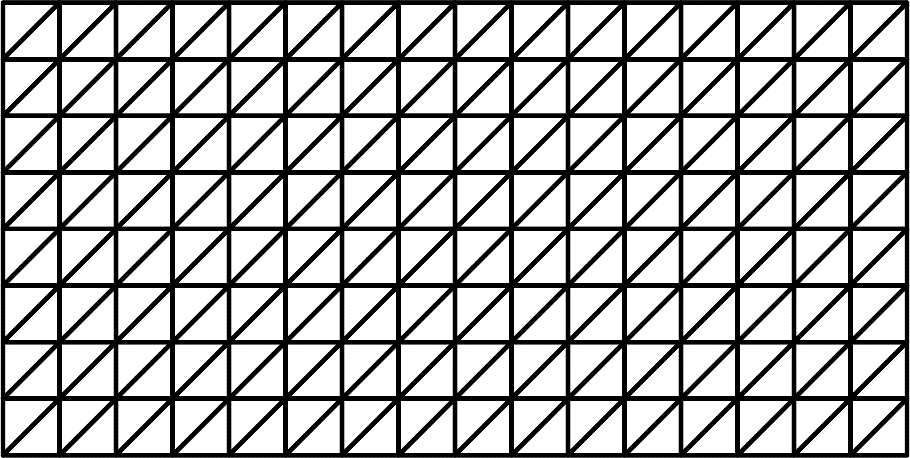}}
\hspace{1em}
\subfloat[Warped mesh]{\includegraphics[width=.35\textwidth]{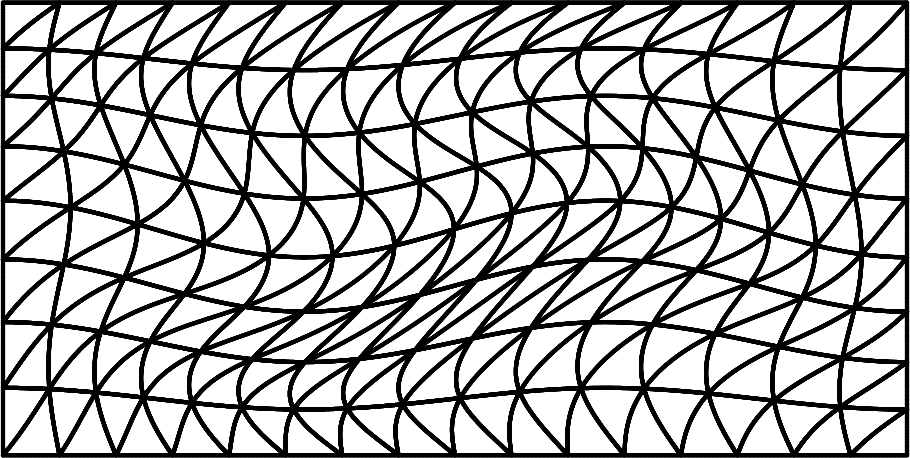}}
\caption{2D curved meshes used for testing entropy conservation and primary conservation.}
\label{fig:warp2d}
\end{figure}

Figure~\ref{fig:dSconverge} shows the evolution of entropy $U(\bm{u})$ over time $[0,T]$, where the final time $T = 2$.  We compare when the entropy-projected conservative variables $\tilde{\bm{u}}$ are computed using the standard $L^2$ projection on the reference element, and when $\tilde{\bm{u}}$ are computed using the weight-adjusted projection.  For the standard $L^2$ projection, the change in entropy does not decrease as $dt$ decreases (for a fixed mesh and order $N$).  When $\tilde{\bm{u}}$ is defined using the weight-adjusted projection, the entropy decreases as $dt$ decreases.  Moreover, the rate of convergence is approximately $O(dt^{4.675})$, which is slightly higher than the expected rate of $O(dt^4)$ when using LSRK-45.  Regardless of whether the standard and weight-adjusted projection was used, the entropy RHS (\ref{eq:entropyrhs}) is $O(10^{-14})$, indicating that the proposed scheme is implemented correctly.  

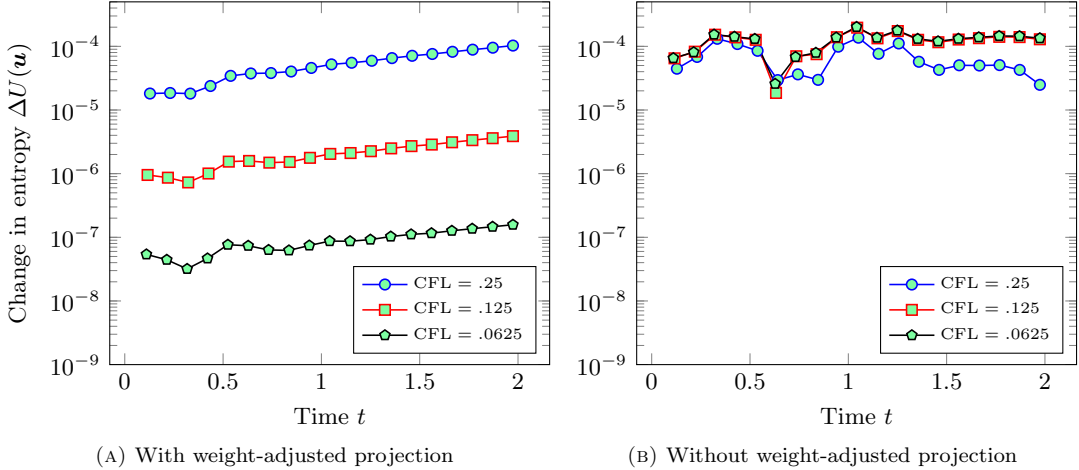
\begin{figure}
\centering
\subfloat[With weight-adjusted projection]{
\begin{tikzpicture}
\begin{semilogyaxis}[
    legend cell align=left,
    legend style={legend pos=south east, font=\tiny},
    width=.45\textwidth,    
    xlabel={Time $t$},
    ylabel={Change in entropy $\Delta U(\bm{u})$}, 
     ymin=1e-9, ymax=5e-4,    
    grid style=dashed,
] 

\addplot[color=blue,mark=*,semithick, mark options={solid,fill=markercolor}]
coordinates{(0.025641,0)(0.128205,1.81455e-05)(0.230769,1.84993e-05)(0.333333,1.81016e-05)(0.435897,2.38212e-05)(0.538462,3.43804e-05)(0.641026,3.75077e-05)(0.74359,3.80509e-05)(0.846154,4.02406e-05)(0.948718,4.57864e-05)(1.05128,5.21023e-05)(1.15385,5.52339e-05)(1.25641,5.93653e-05)(1.35897,6.53204e-05)(1.46154,7.10508e-05)(1.5641,7.58692e-05)(1.66667,8.21148e-05)(1.76923,8.86933e-05)(1.87179,9.52306e-05)(1.97436,0.000102832)};
\addplot[color=red,mark=square*,semithick, mark options={solid,fill=markercolor}]
coordinates{(0.0129032,0)(0.116129,9.54018e-07)(0.219355,8.64409e-07)(0.322581,7.28292e-07)(0.425806,1.00547e-06)(0.529032,1.54507e-06)(0.632258,1.58376e-06)(0.735484,1.48926e-06)(0.83871,1.52599e-06)(0.941935,1.77263e-06)(1.04516,2.04306e-06)(1.14839,2.10785e-06)(1.25161,2.25231e-06)(1.35484,2.49471e-06)(1.45806,2.7075e-06)(1.56129,2.86391e-06)(1.66452,3.11015e-06)(1.76774,3.35711e-06)(1.87097,3.6045e-06)(1.97419,3.88568e-06)};
\addplot[color=black,mark=pentagon*,semithick, mark options={solid,fill=markercolor}]
coordinates{(0.00647249,0)(0.110032,5.38994e-08)(0.213592,4.43651e-08)(0.317152,3.18878e-08)(0.420712,4.65967e-08)(0.524272,7.64772e-08)(0.627832,7.36288e-08)(0.731392,6.33475e-08)(0.834951,6.22881e-08)(0.938511,7.43794e-08)(1.04207,8.71374e-08)(1.14563,8.67872e-08)(1.24919,9.19998e-08)(1.35275,1.02781e-07)(1.45631,1.11204e-07)(1.55987,1.16121e-07)(1.66343,1.26634e-07)(1.76699,1.36562e-07)(1.87055,1.46567e-07)(1.97411,1.57581e-07)};


\legend{${\rm CFL} = .25$,${\rm CFL} = .125$,${\rm CFL} = .0625$ }
\end{semilogyaxis}
\end{tikzpicture}
}
\subfloat[Without weight-adjusted projection]{
\begin{tikzpicture}
\begin{semilogyaxis}[
    legend cell align=left,
    legend style={legend pos=south east, font=\tiny},
    width=.45\textwidth,
    xlabel={Time $t$},
     ymin=1e-9, ymax=5e-4,    
    grid style=dashed,
] 

\addplot[color=blue,mark=*,semithick, mark options={solid,fill=markercolor}]
coordinates{(0.025641,0)(0.128205,4.45681e-05)(0.230769,6.82313e-05)(0.333333,0.000131517)(0.435897,0.000108761)(0.538462,8.51562e-05)(0.641026,2.94791e-05)(0.74359,3.62904e-05)(0.846154,2.97564e-05)(0.948718,9.87886e-05)(1.05128,0.000136806)(1.15385,7.64601e-05)(1.25641,0.000111044)(1.35897,5.72761e-05)(1.46154,4.27e-05)(1.5641,5.03819e-05)(1.66667,4.99789e-05)(1.76923,5.07438e-05)(1.87179,4.26361e-05)(1.97436,2.48694e-05)};
\addplot[color=red,mark=square*,semithick, mark options={solid,fill=markercolor}]
coordinates{(0.0129032,0)(0.116129,6.48664e-05)(0.219355,8.24474e-05)(0.322581,0.000152373)(0.425806,0.000138342)(0.529032,0.000126709)(0.632258,1.85501e-05)(0.735484,6.98555e-05)(0.83871,7.53129e-05)(0.941935,0.00013915)(1.04516,0.000196476)(1.14839,0.000133645)(1.25161,0.000173751)(1.35484,0.000126668)(1.45806,0.000116008)(1.56129,0.000127919)(1.66452,0.0001343)(1.76774,0.000140992)(1.87097,0.00013971)(1.97419,0.0001291)};
\addplot[color=black,mark=pentagon*,semithick, mark options={solid,fill=markercolor}]
coordinates{(0.00647249,0)(0.110032,6.52844e-05)(0.213592,8.08628e-05)(0.317152,0.000153077)(0.420712,0.000141625)(0.524272,0.000130883)(0.627832,2.58046e-05)(0.731392,6.81127e-05)(0.834951,7.92926e-05)(0.938511,0.000137768)(1.04207,0.000201777)(1.14563,0.000136705)(1.24919,0.000177364)(1.35275,0.000131153)(1.45631,0.000119851)(1.55987,0.000131686)(1.66343,0.000138676)(1.76699,0.00014539)(1.87055,0.000144608)(1.97411,0.000134142)};

\legend{${\rm CFL} = .25$,${\rm CFL} = .125$,${\rm CFL} = .0625$ }
\end{semilogyaxis}
\end{tikzpicture}
}
\caption{Change in entropy under an entropy conservative formulation with $N=4$.  In both cases, the magnitude of the entropy RHS (\ref{eq:entropyrhs}) is $O\LRp{10^{-14}}$. }
\label{fig:dSconverge}
\end{figure}

These results suggest that the weight-adjusted projection is necessary to produce an entropy conservative scheme on curvilinear meshes.  However, entropy conservative schemes result in spurious oscillations and lower convergence rates \cite{chan2017discretely}.  In practice, dissipative interface terms are added to produce entropy stable schemes.  In the presence of interface dissipation, the evolution of entropy over time, $L^2$ errors for smooth solutions, and the qualitative behavior of the solution are very similar with or without the weight-adjusted projection.  This may reflect the fact that aliasing-driven instabilities arise from the spatial discretization, and the entropy RHS (\ref{eq:entropyrhs}) is machine precision zero with or without the weight-adjusted projection.  In contrast,  the presence of the weight-adjusted projection affects only the time-derivative on the left-hand side, and may not play as large a role in suppressing aliasing instabilities.

\subsubsection{Local and global conservation}
\label{sec:conservationcheck}
Next, we check that primary conservation is maintained numerically on curved meshes.  We follow \cite{friedrich2017entropy} and examine the semi-discrete evolution of the average of $\bm{u}$ over the domain $\Omega$
\begin{align}
  \pd{}{t}\int_{\Omega} \bm{u} \diff{\bm{x}} = \sum_k \int_{D^k} \pd{\bm{u}}{t} \diff{\bm{x}} = \sum_k \int_{\hat{D}} \pd{\bm{u}}{t} J^k \diff{\hat{\bm{x}}} = \sum_k \bm{1}^T \bm{W}\diag{\bm{J}^k} \bm{V}_q \td{\bm{u}_h}{t}.  
\label{eq:conssemidiscrete}
\end{align}
The quantity (\ref{eq:conssemidiscrete}) is computed using quadrature by multiplying the semi-discrete system (\ref{eq:dgform2}) by $\bm{1}^T\bm{W}\diag{\bm{J}^k} \bm{V}_q$
\begin{align}
&\sum_k \bm{1}^T \bm{W}\diag{\bm{J}^k} \bm{V}_q \td{\bm{u}_h}{t}= \nonumber\\
& \underbrace{\bm{1}^T\bm{W}\diag{\bm{J}^k}\bm{V}_q\LRp{
 \LRs{\begin{array}{cc}\tilde{\bm{P}}^k_q & \tilde{\bm{L}}^k_q\end{array}} \sum_{j=1}^d \LRp{2\bm{D}^j_k \circ \bm{F}_{j,S}}\bm{1} + \sum_{j=1}^d \tilde{\bm{L}}^k_q \diag{\bm{n}_j}\LRp{\bm{f}_j^* - \bm{f}_j(\tilde{\bm{u}}_f)}}}_{\text{conservation residual}}.
 \label{eq:consres}
\end{align}

We consider three different cases:
\begin{enumerate}
\item \textbf{Unmodified:} the weight-adjusted DG scheme (\ref{eq:dgform2}) is used without any special modifications,
\item \textbf{Polynomial approximation:} the weight-adjusted DG scheme (\ref{eq:dgform2}) utilizes the $L^2$ projection of $J^k$ onto $P^N(\hat{D})$ in the weight-adjusted mass matrix (such that Lemma~\ref{eq:conscorrect1} holds), 
\item \textbf{Conservative correction:} the conservation correction (\ref{eq:conscorrect}) is applied to the right hand side of the scheme (\ref{eq:dgform2}).  
\end{enumerate}
All results use a Lax-Friedrichs flux, a CFL of $1/2$, $N=4$, and the curved mesh warping shown in Figure~\ref{fig:warp2d}.  We use volume quadrature rules from \cite{xiao2010quadrature} which are exact for degree $2N+1$ polynomials, and use $(N+1)$-node 1D Gauss--Legendre quadrature rules on the faces.  

Figure~\ref{subfig:cons1} shows computed conservation residuals from (\ref{eq:consres}) on the curved mesh in Figure~\ref{fig:warp2d} for the discontinuous pulse initial condition (\ref{eq:discontin}).  The conservation residual oscillates around $O(10^{-6})$ for the unmodified scheme (\ref{eq:dgform2}).  Modifying the scheme using either the ``Polynomial approximation'' or ``Conservative correction'' approaches reduces the conservation error to between $O(10^{-11})$ and $O(10^{-12})$.  
We have also studied the accuracy of each case in by comparing $L^2$ errors on a curved mesh for a smooth vortex solution  at time $T=5$.  We observe that in all cases, the $L^2$ errors are virtually identical.  For $N=4$ on an $8\times 8$ mesh, the $L^2$ errors for each case differ only in the 5th digit, while for the $16\times 16$ and $32\times 32$ meshes, they differ only in the 8th digit.  This confirms that both approaches outlined in Section~\ref{sec:conservation} restore primary conservation with no perceivable effect on accuracy.  


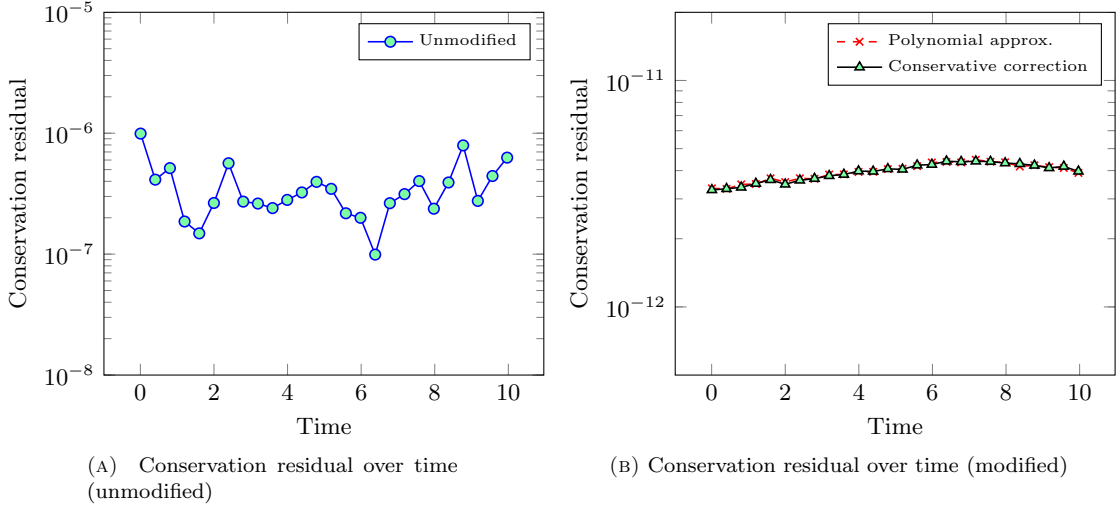
\begin{figure}
\centering
\subfloat[ Conservation residual over time (unmodified)]{
\begin{tikzpicture}
\begin{semilogyaxis}[
    legend cell align=left,
    legend style={legend pos=north east, font=\tiny},
    width=.45\textwidth,    
    xlabel={Time},
    ylabel={Conservation residual}, 
    ymin=1e-8, ymax=1e-5,    
    grid style=dashed,
] 

\addplot[color=blue,mark=*,semithick, mark options={solid,fill=markercolor}]
coordinates{(0.00642674,9.9266e-07)(0.404884,4.12992e-07)(0.803342,5.14326e-07)(1.2018,1.86091e-07)(1.60026,1.48526e-07)(1.99871,2.64763e-07)(2.39717,5.64483e-07)(2.79563,2.71857e-07)(3.19409,2.62073e-07)(3.59254,2.40249e-07)(3.991,2.80807e-07)(4.38946,3.23725e-07)(4.78792,3.9606e-07)(5.18638,3.46717e-07)(5.58483,2.17698e-07)(5.98329,1.99781e-07)(6.38175,9.91165e-08)(6.78021,2.63862e-07)(7.17866,3.13545e-07)(7.57712,4.01892e-07)(7.97558,2.37462e-07)(8.37404,3.91475e-07)(8.77249,7.92905e-07)(9.17095,2.75479e-07)(9.56941,4.42523e-07)(9.96787,6.28215e-07)};

%
\legend{Unmodified}
\end{semilogyaxis}
\end{tikzpicture}
\label{subfig:cons1}}
\subfloat[Conservation residual over time (modified)]{
\begin{tikzpicture}
\begin{semilogyaxis}[
    legend cell align=left,
    legend style={legend pos=north east, font=\tiny},
    width=.45\textwidth,    
    xlabel={Time},
    ylabel={Conservation residual}, 
    ymin=.5e-12, ymax=2e-11,    
    grid style=dashed,
    ytick={1e-12, 1e-11},
    yticklabels={{$ 10^{-12}$},{$ 10^{-11}$}}
] 


\addplot[color=red,mark=x,dashed,semithick, mark options={solid,fill=markercolor}]
coordinates{(0.00642674,3.34034e-12)(0.404884,3.31953e-12)(0.803342,3.47682e-12)(1.2018,3.48782e-12)(1.60026,3.69874e-12)(1.99871,3.55961e-12)(2.39717,3.69915e-12)(2.79563,3.67725e-12)(3.19409,3.84246e-12)(3.59254,3.84801e-12)(3.991,3.9449e-12)(4.38946,3.95075e-12)(4.78792,4.05986e-12)(5.18638,4.07415e-12)(5.58483,4.18282e-12)(5.98329,4.34447e-12)(6.38175,4.3659e-12)(6.78021,4.34614e-12)(7.17866,4.46538e-12)(7.57712,4.37571e-12)(7.97558,4.35258e-12)(8.37404,4.16896e-12)(8.77249,4.22267e-12)(9.17095,4.14187e-12)(9.56941,4.10735e-12)(9.96787,3.90628e-12)};

\addplot[color=black,mark=triangle*,semithick, mark options={solid,fill=markercolor}]
coordinates{(0.00642674,3.28579e-12)(0.404884,3.32419e-12)(0.803342,3.37232e-12)(1.2018,3.50534e-12)(1.60026,3.64225e-12)(1.99871,3.47872e-12)(2.39717,3.62415e-12)(2.79563,3.68648e-12)(3.19409,3.79388e-12)(3.59254,3.84432e-12)(3.991,3.97812e-12)(4.38946,3.96804e-12)(4.78792,4.0684e-12)(5.18638,4.04986e-12)(5.58483,4.22413e-12)(5.98329,4.25356e-12)(6.38175,4.39801e-12)(6.78021,4.37474e-12)(7.17866,4.39737e-12)(7.57712,4.38242e-12)(7.97558,4.31792e-12)(8.37404,4.28398e-12)(8.77249,4.21688e-12)(9.17095,4.10925e-12)(9.56941,4.17163e-12)(9.96787,3.96996e-12)};
\legend{Polynomial approx., Conservative correction }
\end{semilogyaxis}
\end{tikzpicture}
\label{subfig:cons2}}
%
%
%
%
%
\caption{Conservation residuals for a discontinuous initial condition using the unmodified WADG formulation (\ref{eq:dgform2}) and two different techniques for restoring local conservation.  Both experiments use a Lax-Friedrichs flux, a CFL of $1/2$, $N=4$, and the curved mesh warping shown in Figure~\ref{fig:warp2d}.  }
\label{fig:cons}
\end{figure}

\subsubsection{Accuracy and convergence}

\begin{figure}
\centering
\subfloat[Affine mesh]{\includegraphics[width=.4\textwidth]{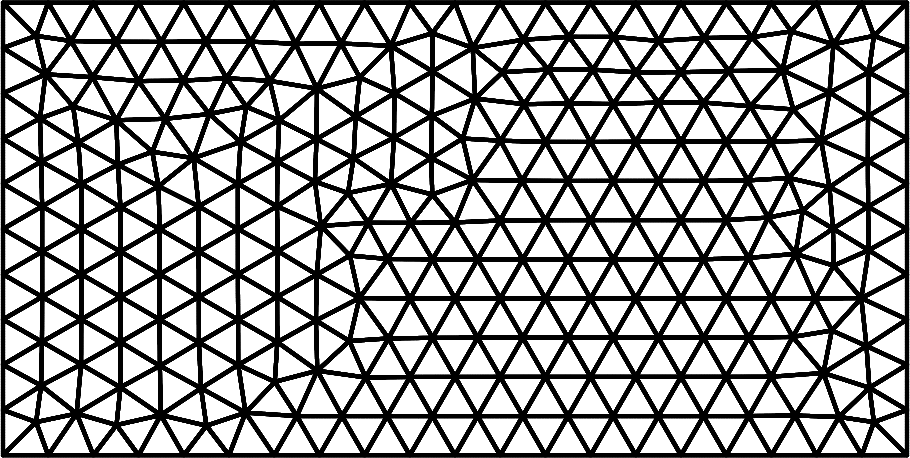}}
\hspace{2em}
\subfloat[Curved mesh]{\includegraphics[width=.4\textwidth]{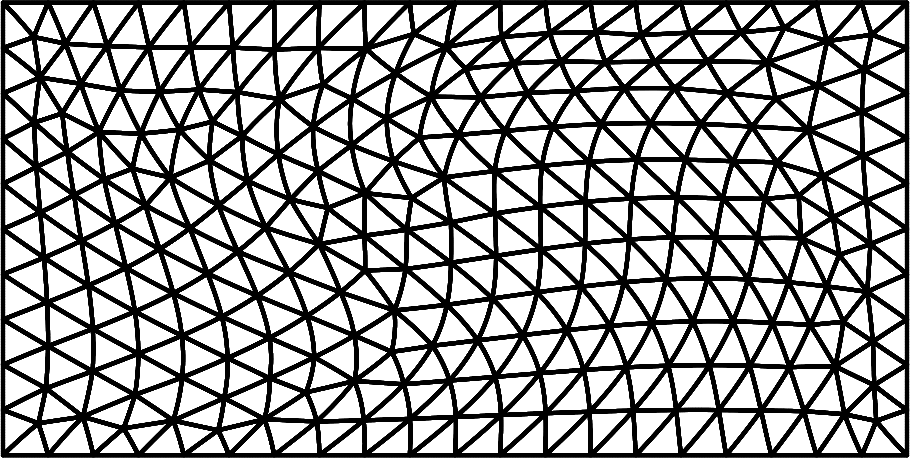}}
\caption{Example of affine and warped meshes ($h = 1$ shown) used in convergence studies.}
\label{fig:warp2dconverge}
\end{figure}

Finally, we test the accuracy of the proposed schemes for smooth solutions on curved meshes in two dimensions.  
We use the isentropic vortex problem \cite{shu1998essentially}, which has an analytical solution 
\begin{align}
\rho(\bm{x},t) &= \LRp{1 - \frac{\frac{1}{2}(\gamma-1)(\beta e^{1-r(\bm{x},t)^2})^2}{8\gamma \pi^2}}^{\frac{1}{\gamma-1}}, \qquad p = \rho^{\gamma},\\
u(\bm{x},t) &= 1 - \frac{\beta}{2\pi} e^{1-r(\bm{x},t)^2}(x_2-c_2), \qquad v(\bm{x},t) = \frac{\beta}{2\pi} e^{1-r(\bm{x},t)^2}(x_2-c_2),\nonumber
\end{align}
where $u, v$ are the $x_1$ and $x_2$ velocity and $r(\bm{x},t) = \sqrt{(x_1-c_1-t)^2 + (x_2-c_2)^2}$.  Here, we take $c_1 = 5, c_2 = 0$ and $\beta = 5$.  
The solution is computed on a periodic rectangular domain $[0, 20] \times [-5,5]$ at final time $T=5$.  Quasi-uniform triangular meshes are generated using \textsc{GMSH} \cite{geuzaine2009gmsh}, and a curvilinear warping is applied to the mesh to test the effect of non-affine mappings.  This warping is shown in Figure~\ref{fig:warp2dconverge} and is defined by mapping nodal positions on each triangle to warped nodal positions $(\tilde{x_1},\tilde{x_2})$ via
\begin{align*}
\tilde{x_1} &= x_1 + \sin\LRp{\pi x_1 / 20}\sin\LRp{2\pi (x_2+5)/10}\\
\tilde{x_2} &= x_2 - \frac{1}{2}\sin\LRp{2 \pi x_1 / 20}\sin\LRp{\pi (x_2+5)/10}.
\end{align*}  
To ensure primary conservation, we use the ``Polynomial Approximation'' strategy described in Section~\ref{sec:conservation} and compute the inverse of the weight-adjusted mass matrix using the degree $N$ $L^2$ projection $\hat{\Pi}_N J^k$ instead of $J^k$.  The computed $L^2$ errors are shown in Figure~\ref{fig:converge2d}.  We observe optimal $O(h^{N+1})$ rates of convergence for $N=2, N= 4$.  For degree $N=3$, the rate of convergence is slightly higher than $O(h^{4})$ and may indicate that the mesh is not yet sufficiently fine for the $L^2$ error to show the asymptotic convergence rate.  

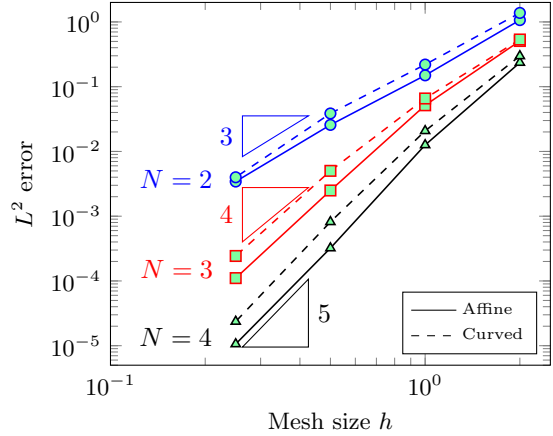
\begin{figure}
\centering
\begin{tikzpicture}
\begin{loglogaxis}[
    legend cell align=left,
    legend style={legend pos=south east, font=\tiny},
    width=.45\textwidth,    
    xlabel={Mesh size $h$},
    ylabel={$L^2$ error}, 
     ymin=5e-6, ymax=2,    
     xmin=1e-1, xmax=2.5,         
    grid style=dashed,
    legend entries={Affine,Curved}
] 
\addlegendimage{no markers,black}
\addlegendimage{no markers,dashed,black}

\addplot[color=blue,mark=*,semithick, mark options={solid,fill=markercolor}]
coordinates{(2,1.06717)(1,0.149639)(0.5,0.025693)(0.25,0.00342827)} [yshift=4pt] node[left, pos=1.05, color=blue] {$N = 2$};
\addplot[color=blue,mark=*,dashed,semithick, mark options={solid,fill=markercolor}]
coordinates{(2,1.37709)(1,0.219501)(0.5,0.0385896)(0.25,0.0039906)};
\logLogSlopeTriangleFlip{0.45}{0.15}{0.575}{3}{blue}

\addplot[color=red,mark=square*,semithick, mark options={solid,fill=markercolor}]
coordinates{(2,0.50782)(1,0.0516053)(0.5,0.00249425)(0.25,0.000110618)}[yshift=8pt] node[left, pos=1.05, color=red] {$N = 3$};
\addplot[color=red,mark=square*,dashed,semithick, mark options={solid,fill=markercolor}]
coordinates{(2,0.535036)(1,0.0657418)(0.5,0.00501536)(0.25,0.000243005)};
\logLogSlopeTriangleFlip{0.45}{0.15}{0.34}{4}{red}

\addplot[color=black,mark=triangle*,semithick, mark options={solid,fill=markercolor}]
coordinates{(2,0.235374)(1,0.0125714)(0.5,0.000321559)(0.25,1.06097e-05)} [yshift=8pt] node[left, pos=1.05, color=black] {$N = 4$};
\addplot[color=black,mark=triangle*,dashed,semithick, mark options={solid,fill=markercolor}]
coordinates{(2,0.297758)(1,0.0207604)(0.5,0.000816006)(0.25,2.36102e-05)};
\logLogSlopeTriangle{0.45}{0.15}{0.05}{5}{black}

\end{loglogaxis}
\end{tikzpicture}
\caption{Convergence of $L^2$ errors for the 2D isentropic vortex problem on affine and curved meshes.}
\label{fig:converge2d}
\end{figure}


\subsection{Three dimensional compressible Euler equations}

In three dimensions, the compressible Euler equations are given by
\begin{align}
\pd{}{t}\LRp{\begin{array}{c}
\rho\\
\rho u\\
\rho v\\
\rho w\\
E
\end{array}} +
\pd{}{x_1}\LRp{\begin{array}{c}
\rho u\\
\rho u^2+p\\
\rho uv\\
\rho uw\\
u(E+p)
\end{array}} + \pd{}{x_2}\LRp{\begin{array}{c}
\rho v\\
\rho uv\\
\rho v^2+p\\
\rho vw\\
v(E+p)
\end{array}}  + \pd{}{x_3}\LRp{\begin{array}{c}
\rho w\\
\rho uw\\
\rho vw\\
\rho w^2+p\\
w(E+p)
\end{array}} &= 0,
\label{eq:euler3d}
\end{align}
where the pressure $p$ and specific internal energy $\rho e$ are defined 
\begin{align}
p = (\gamma-1)\LRp{E - \frac{1}{2}\rho (u^2+v^2+w^2)}, \qquad \rho e = E - \frac{1}{2}\rho (u^2+v^2+w^2).  
\label{eq:pressure3d}
\end{align}
The formula for the entropy $U(\bm{u})$ in three dimensions is the same as the two-dimensional formula (\ref{eq:entropy2d}).  
 The entropy variables in three dimensions are 
\begin{align}
v_1 = \frac{\rho e (\gamma + 1 - s) - E}{\rho e}, \qquad v_2 = \frac{\rho u}{\rho e}, \qquad v_3 = \frac{\rho v}{\rho e}, \qquad v_4 = \frac{\rho w}{\rho e}, \qquad v_5 = -\frac{\rho}{\rho e}.
\end{align}
The conservation variables in terms of the entropy variables are given by
\begin{equation}
\rho = -(\rho e) v_5, \qquad \rho u = (\rho e) v_2, \qquad \rho v = (\rho e) v_3, \qquad \rho w = (\rho e) v_4, \qquad E = (\rho e)\LRp{1 - \frac{{v_2^2+v_3^2+v_4^2}}{2 v_5}},
\end{equation}
where $\rho e$ and $s$ in terms of the entropy variables are 
\begin{equation}
\rho e = \LRp{\frac{(\gamma-1)}{\LRp{-v_5}^{\gamma}}}^{1/(\gamma-1)}e^{\frac{-s}{\gamma-1}}, \qquad s = \gamma - v_1 + \frac{{v_2^2+v_3^2+v_4^2}}{2v_5}.
\end{equation}
A set of entropy conservative numerical fluxes for the three-dimensional compressible Euler equations can be written as
\begin{gather}
\bm{f}_{1,S} = \LRp{\begin{array}{c}
\avg{\rho}^{\log}\avg{u}\\
\avg{\rho}^{\log}\avg{u}^2 + p_{\rm avg}\\
\avg{\rho}^{\log}\avg{u}\avg{v}\\
\avg{\rho}^{\log}\avg{u}\avg{w}\\
(E_{\rm avg}+ p_{\rm avg})\avg{u}\\
\end{array}}, 
\qquad 
\bm{f}_{2,S} = \LRp{\begin{array}{c}
\avg{\rho}^{\log}\avg{v}\\
\avg{\rho}^{\log}\avg{u}\avg{v}\\
\avg{\rho}^{\log}\avg{v}^2 + p_{\rm avg}\\
\avg{\rho}^{\log}\avg{v}\avg{w}\\
(E_{\rm avg}+ p_{\rm avg})\avg{v}\\
\end{array}},\\
\bm{f}_{3,S} = \LRp{\begin{array}{c}
\avg{\rho}^{\log}\avg{w}\\
\avg{\rho}^{\log}\avg{u}\avg{w}\\
\avg{\rho}^{\log}\avg{v}\avg{w}\\
\avg{\rho}^{\log}\avg{w}^2 + p_{\rm avg}\\
(E_{\rm avg}+ p_{\rm avg})\avg{w}\\
\end{array}}.\nonumber
\end{gather}
where we have defined the auxiliary quantities
\begin{align}
p_{\rm avg} &= \frac{\avg{\rho}}{2\avg{\beta}}, \qquad E_{\rm avg} = \frac{\avg{\rho}^{\log}}{2(\gamma-1)\avg{\beta}^{\log}} + \frac{1}{2}\avg{\rho}^{\log}\nor{\bm{u}}^2_{\rm avg}\\
\nor{\bm{u}}^2_{\rm avg} &= 2(\avg{u}^2 + \avg{v}^2 + \avg{w}^2) - \LRp{\avg{u^2} +\avg{v^2} + \avg{w^2}}.\nonumber
\end{align}

\subsubsection{Accuracy and convergence}

As before, we test the accuracy of the proposed scheme using an isentropic vortex solution adapted to three dimensions.  We take the solution to be the extruded 2D vortex propagating in the $x_2$ direction, whose analytic expression is derived from \cite{williams2013nodal}
\begin{align*}
\rho(\bm{x},t) &= \LRp{1-\frac{(\gamma-1)}{2}\Pi^2}^{\frac{1}{\gamma-1}}\\
\bm{u}(\bm{x},t) &= \Pi \bm{r}, \\
E(\bm{x},t) &= \frac{p_0}{\gamma-1}\LRp{1-\frac{\gamma-1}{2}\Pi^2}^{\frac{\gamma}{\gamma-1}} + \frac{\rho}{2}\LRb{\bm{u}}.
\end{align*}
where $\bm{u} = (u,v,w)^T$ is the velocity vector and 
\[
\Pi = \Pi_{\max}e^{\frac{1-\bm{r}^T\bm{r}}{2}}, \qquad \bm{r} = \begin{pmatrix}
-(x_2-c_2-t)\\
x_1-c_1\\
0
\end{pmatrix}.
\]
In this problem, we take $c_1 = c_2 = 5$, $p_0 = {1}/{\gamma}$, and $\Pi_{\max} = 0.4$.  The problem is solved on the domain $[0,10]\times [0,20]\times [0,10]$.  
We use \textsc{GMSH} to construct three unstructured meshes consisting of 1354, 9543, and 72923 affine tetrahedra, corresponding to $h = 2$, $h = 1$, and $h = 1/2$ (shown in Figure~\ref{subfig:mesh3d}).  We apply a curvilinear warping by mapping nodal positions each tetrahedron to warped nodal positions $(\tilde{x_1},\tilde{x_2},\tilde{x_3})$ via
\begin{align*}
\tilde{x_1} &= x_1 + \frac{1}{2}\sin\LRp{\pi \frac{x_1}{10}}\sin\LRp{2\pi \frac{x_2}{20}}\sin\LRp{\pi \frac{x_3}{10}},\\
\tilde{x_2} &= x_2 - \sin\LRp{2 \pi \frac{x_1}{10}}\sin\LRp{\pi \frac{x_2}{10}}\sin\LRp{2\pi \frac{x_3}{10}},\\
\tilde{x_3} &= x_3 + \frac{1}{2}\sin\LRp{\pi \frac{x_1}{10}}\sin\LRp{2\pi \frac{x_2}{10}}\sin\LRp{\pi \frac{x_3}{10}}.
\end{align*} 
The geometric terms are approximated using (\ref{eq:iconscurl2}) to ensure the satisfaction of the discrete GCL.  Figure~\ref{fig:converge3d} shows $L^2$ errors at final time $T = 5$ for $N = 2,3,4$ on both affine and curved meshes.  In both cases, we observe optimal  $O(h^{N+1})$ asymptotic rates of convergence for $N = 2,3$, while for $N = 4$ we observe a rate which is slightly higher than $O(h^{N+1/2})$ but not quite $O(h^{N+1})$.  This may be due to the fact that the time-stepper is $4$th order while the spatial discretization is $5$th order.

\begin{figure}
\centering
\subfloat[Mesh for $h = 1/2$]{\raisebox{2em}{\includegraphics[width=.4\textwidth]{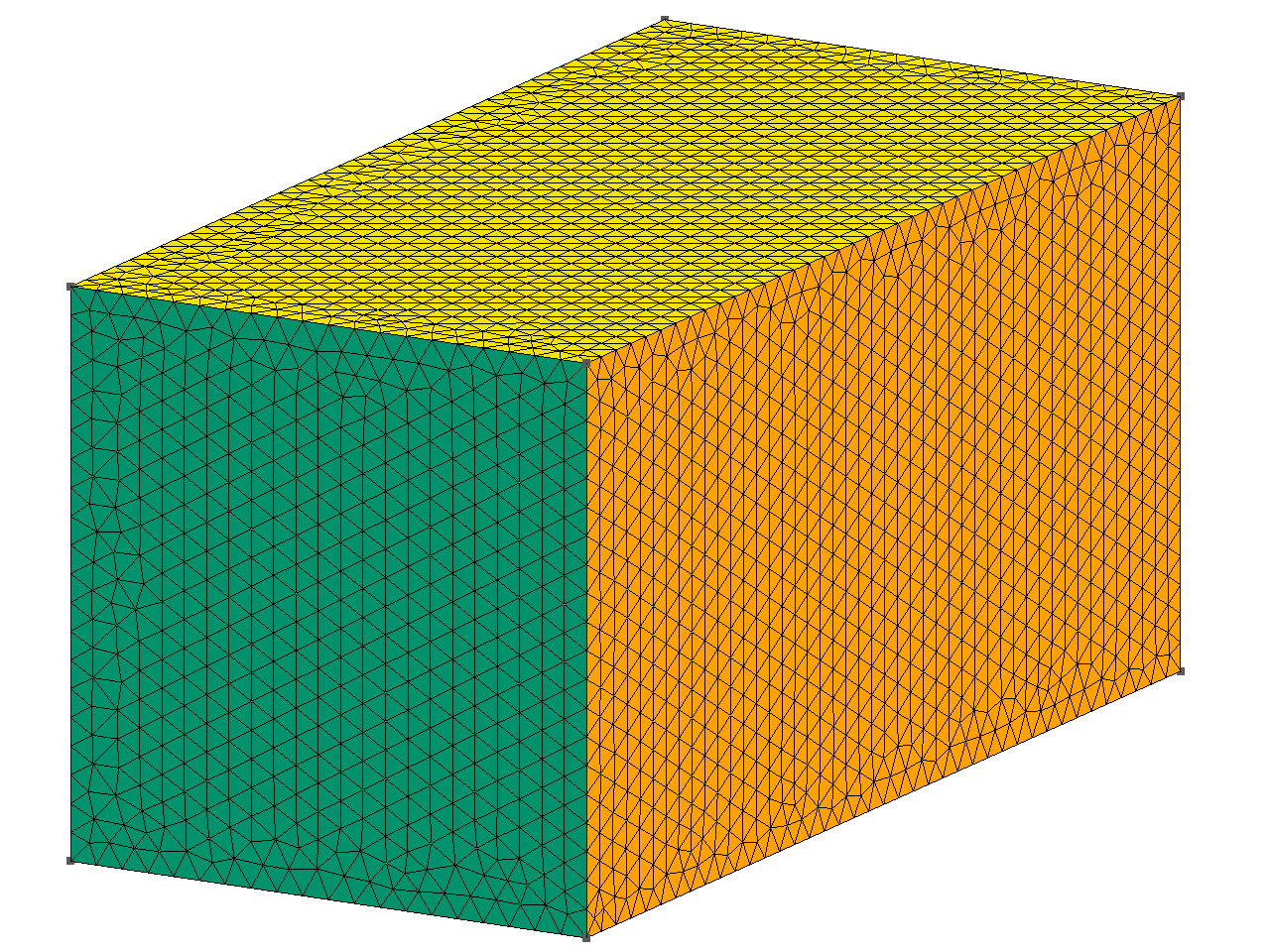}}\label{subfig:mesh3d}}
\subfloat[$L^2$ errors]{\begin{tikzpicture}
\begin{loglogaxis}[
    legend cell align=left,
    legend style={legend pos=south east, font=\tiny},
    width=.45\textwidth,    
    xlabel={Mesh size $h$},
    ylabel={$L^2$ error}, 
     ymin=1e-4, ymax=2,    
     xmin=2.5e-1, xmax=2.5,         
    grid style=dashed,
    legend entries={Affine,Curved}
] 
\addlegendimage{no markers,black}
\addlegendimage{no markers,dashed,black}

\addplot[color=blue,mark=*,semithick, mark options={solid,fill=markercolor}]
coordinates{(2,1.05519)(1,0.143515)(0.5,0.0212682)}[yshift=4pt] node[left, pos=1.05, color=blue] {$N = 2$};
\addplot[color=blue,mark=*,dashed,semithick, mark options={solid,fill=markercolor}]
coordinates{(2,1.20915)(1,0.2069)(0.5,0.0284505)};
\logLogSlopeTriangleFlip{0.45}{0.15}{0.6}{3}{blue}

\addplot[color=red,mark=square*,semithick, mark options={solid,fill=markercolor}]
coordinates{(2,0.339318)(1,0.0314342)(0.5,0.00197699)}[yshift=8pt] node[left, pos=1.05, color=red] {$N = 3$};
\addplot[color=red,mark=square*,dashed,semithick, mark options={solid,fill=markercolor}]
coordinates{(2,0.464613)(1,0.0513369)(0.5,0.00334595)};
\logLogSlopeTriangleFlip{0.45}{0.15}{0.375}{4}{red}

\addplot[color=black,mark=triangle*,semithick, mark options={solid,fill=markercolor}]
coordinates{(2,0.122229)(1,0.00488434)(0.5,0.000192453)}[yshift=8pt] node[left, pos=1.05, color=black] {$N = 4$};
\addplot[color=black,mark=triangle*,dashed,semithick, mark options={solid,fill=markercolor}]
coordinates{(2,0.184547)(1,0.0104361)(.5,0.0003960681)};
\logLogSlopeTriangle{0.45}{0.15}{0.04}{5}{black}

\end{loglogaxis}
\end{tikzpicture}}
\caption{Convergence of $L^2$ errors for the 3D isentropic vortex problem on unstructured affine and curved meshes, with optimal $O(h^{N+1})$ rates of convergence shown for reference.}
\label{fig:converge3d}
\end{figure}
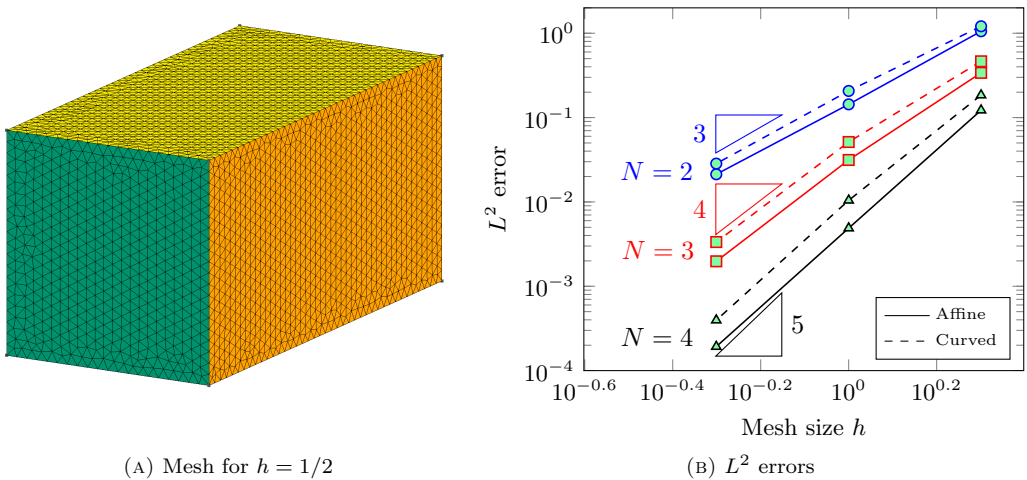

We also compared approximations of geometric factors using degree $N$ polynomials via (\ref{eq:iconscurl2}) and degree $N-1$ polynomials via (\ref{eq:iconscurl}).  For the isentropic vortex problem on warped meshes, the approximation of the geometric factors does not impact accuracy significantly.  Utilizing (\ref{eq:iconscurl}) and approximating geometric factors with degree $N-1$ polynomials only changed the error in the $5$th significant digit, and resulted in the same asymptotic rates of convergence on curved 3D meshes.  Future work will explore the effect of a more accurate geometric approximation on curved boundaries where a solid wall boundary condition is applied, where the approximation of geometry has been shown to have a more significant effect  \cite{toulorge2016optimizing}.  


\subsubsection{Inviscid Taylor--Green vortex}

Our last numerical experiment investigates the behavior of entropy stable DG schemes for the inviscid Taylor--Green vortex \cite{ae1937mechanism, gassner2016split, crean2018entropy}.  The domain is the periodic box $[-\pi,\pi]^3$, and the initial conditions are given as
\begin{align*}
\rho &= 1\\
u &= \sin(x_1)\cos(x_2)\cos(x_3),\\
v &= -\cos(x_1)\sin(x_2)\cos(x_3),\\ 
w &= 0,\\
p &= \frac{100}{\gamma} + \frac{1}{16} \LRp{\cos(2x_1) + \cos(2x_2)}\LRp{2+\cos(2x_3)}.
\end{align*}
The Taylor--Green vortex is used to study the transition and decay of turbulence \cite{debonis2013solutions}.  In the absence of viscosity, the Taylor--Green vortex develops smaller and smaller scales, implying that for sufficiently large times, the solution contains under-resolved features.  We study the evolution of the kinetic energy $\kappa(t)$ 
\[
\kappa(t) =\frac{1}{\LRb{\Omega}} \int_{\Omega} \rho \bm{u}\cdot\bm{u} \diff{\bm{x}},
\]
as well as the kinetic energy dissipation rate $-\pd{\kappa}{t}$, which we approximate by differencing $\kappa(t)$.  Figure~\ref{fig:tg} shows the evolution of $\kappa(t)$ over time for both affine and curvilinear meshes with $h = \pi/8$ and $N = 3$.  We also use a curved coarse mesh with element size $h = \pi$ test the convergence of the average entropy for a non-dissipative entropy conservative formulation as the timestep decreases.  In both cases, the mesh is defined by constructing nodal positions $\tilde{\bm{u}}$ from warpings of affine nodal positions $\bm{x}$ as follows 
\[
\tilde{\bm{x}} = \bm{x} + .125 \sin(x_1)\sin(x_2)\sin(x_3).
\]
The kinetic energy is plotted at 100 equally spaced times between $[0,20]$.  The simulation is stable and does not blow up despite the lack of filtering, limiting, or artificial viscosity.  The results are qualitatively similar to those reported in the literature, with the kinetic energy dissipation rate increasing around $t = 4$ and peaking with a value of roughly $0.014$ before $t = 9$ \cite{debonis2013solutions, gassner2016split}.  We also observe that, for an entropy conservative formulation, the average entropy appears to converge at a rate of $O(dt^{5})$ for a 4th order time-stepping method.  The same phenomena was also observed in \cite{chan2017discretely, crean2018entropy}.  

\begin{figure}
\centering
\subfloat[KE dissipation rate for $N=3$, $h = \pi/8$]{
\begin{tikzpicture}
\begin{axis}[
        scaled ticks=false, 
        tick label style={/pgf/number format/fixed},
	legend cell align=left,
	legend style={font=\tiny},
	width=.475\textwidth,
    xlabel={Time $t$},
    ylabel={$-\pd{\kappa}{t}$},
ymin=-.0025, ymax=.017,
    legend pos=north east,
    xmajorgrids=true,
    ymajorgrids=true,
    grid style=dashed,
    ytick={0, .005, .01, .015},
    yticklabels={0, .005, .01, .015}    
] 
\addplot[color=blue,semithick, mark options={fill=markercolor}]
coordinates{(0.007149,-0)(0.207328,-1.12783e-05)(0.407507,-1.46618e-05)(0.607685,-2.53763e-05)(0.807864,-2.98876e-05)(1.00804,-2.65041e-05)(1.20822,-3.15793e-05)(1.4084,-2.48123e-05)(1.60858,-1.97371e-05)(1.80876,-2.42484e-05)(2.00894,-2.0301e-05)(2.20912,-1.40979e-05)(2.40929,-6.767e-06)(2.60947,1.12783e-05)(2.80965,2.98876e-05)(3.00983,6.31587e-05)(3.21001,0.000110528)(3.41019,0.000181581)(3.61037,0.000329891)(3.81054,0.00058027)(4.01072,0.000940613)(4.2109,0.00144081)(4.41108,0.00207803)(4.61126,0.00279816)(4.81144,0.00343031)(5.01162,0.00381377)(5.2118,0.00407881)(5.41197,0.00431791)(5.61215,0.00455532)(5.81233,0.00493822)(6.01251,0.00546492)(6.21269,0.00609087)(6.41287,0.00687415)(6.61305,0.00773299)(6.81323,0.00853827)(7.0134,0.00933057)(7.21358,0.010159)(7.41376,0.0111904)(7.61394,0.0123673)(7.81412,0.0132419)(8.0143,0.013865)(8.21448,0.0140928)(8.41466,0.0139739)(8.61483,0.0137714)(8.81501,0.0135272)(9.01519,0.0133541)(9.21537,0.0133011)(9.41555,0.0131477)(9.61573,0.0127451)(9.81591,0.012215)(10.0161,0.0117295)(10.2163,0.0113731)(10.4164,0.0110065)(10.6166,0.0106795)(10.8168,0.0103857)(11.017,0.0100281)(11.2172,0.00961873)(11.4173,0.00931591)(11.6175,0.00903677)(11.8177,0.0087441)(12.0179,0.00844409)(12.2181,0.00816101)(12.4182,0.0078689)(12.6184,0.00757115)(12.8186,0.00726438)(13.0188,0.00696776)(13.2189,0.00667621)(13.4191,0.00640158)(13.6193,0.00614162)(13.8195,0.0058918)(14.0197,0.00565214)(14.2198,0.00541755)(14.42,0.00519085)(14.6202,0.0049839)(14.8204,0.00477863)(15.0206,0.0045728)(15.2207,0.00437148)(15.4209,0.00419272)(15.6211,0.00402693)(15.8213,0.00386339)(16.0214,0.00371678)(16.2216,0.00359948)(16.4218,0.00349347)(16.622,0.00337222)(16.8222,0.00324083)(17.0223,0.00311451)(17.2225,0.00299835)(17.4227,0.00289571)(17.6229,0.00280323)(17.8231,0.00270906)(18.0232,0.00261432)(18.2234,0.00252071)(18.4236,0.00243612)(18.6238,0.00235717)(18.8239,0.00227371)(19.0241,0.00218913)(19.2243,0.00211187)(19.4245,0.00204025)(19.6247,0.00197089)(19.8248,0.00190548)};

\addplot[color=red, dashed,semithick, mark options={fill=markercolor}]
coordinates{(0.00662,-0)(0.20523,-1.01505e-05)(0.40384,-1.24062e-05)(0.60245,-2.31206e-05)(0.801059,-2.59402e-05)(0.999669,-2.31206e-05)(1.19828,-2.81958e-05)(1.39689,-2.19928e-05)(1.5955,-1.80453e-05)(1.79411,-2.14288e-05)(1.99272,-1.86093e-05)(2.19133,-1.63536e-05)(2.38994,-8.45875e-06)(2.58855,5.07525e-06)(2.78716,2.0301e-05)(2.98577,4.96247e-05)(3.18438,9.4738e-05)(3.38299,0.000160152)(3.5816,0.000283086)(3.78021,0.000504706)(3.97881,0.000825574)(4.17743,0.00126543)(4.37603,0.00183499)(4.57464,0.00247672)(4.77325,0.0030897)(4.97186,0.00349741)(5.17047,0.00375061)(5.36908,0.00398971)(5.56769,0.00419893)(5.7663,0.00449555)(5.96491,0.00496811)(6.16352,0.00551849)(6.36213,0.00618335)(6.56074,0.00695535)(6.75935,0.00776796)(6.95796,0.00847567)(7.15657,0.00922568)(7.35518,0.0102655)(7.55379,0.0113212)(7.7524,0.0122049)(7.95101,0.0127772)(8.14962,0.0130801)(8.34823,0.0132329)(8.54684,0.0131302)(8.74545,0.012815)(8.94406,0.0124564)(9.14267,0.0121327)(9.34128,0.0117836)(9.53989,0.0115276)(9.7385,0.0112761)(9.93711,0.0109236)(10.1357,0.0105819)(10.3343,0.0102689)(10.5329,0.0100225)(10.7315,0.00976309)(10.9302,0.00945689)(11.1288,0.00915857)(11.3274,0.00887436)(11.526,0.00861327)(11.7246,0.00831947)(11.9232,0.00802623)(12.1218,0.00775273)(12.3204,0.00746513)(12.519,0.00717584)(12.7176,0.00690742)(12.9163,0.00662038)(13.1149,0.00633448)(13.3135,0.00608184)(13.5121,0.00583767)(13.7107,0.00559631)(13.9093,0.00538653)(14.1079,0.00517958)(14.3065,0.00497093)(14.5051,0.00477976)(14.7037,0.00457393)(14.9024,0.00436246)(15.101,0.00417242)(15.2996,0.00400832)(15.4982,0.00385719)(15.6968,0.00370663)(15.8954,0.00355042)(16.094,0.00340831)(16.2926,0.00329102)(16.4912,0.00318557)(16.6898,0.00307729)(16.8884,0.00296056)(17.0871,0.00284214)(17.2857,0.00272597)(17.4843,0.00261714)(17.6829,0.00252353)(17.8815,0.00243725)(18.0801,0.00234984)(18.2787,0.00226413)(18.4773,0.00218349)(18.6759,0.0021051)(18.8745,0.00202954)(19.0732,0.00195848)(19.2718,0.00189589)(19.4704,0.00184288)(19.669,0.00179438)(19.8676,0.00174476)};

\legend{Affine, Curved}
\end{axis}\end{tikzpicture}
}
\subfloat[Entropy conservative convergence of $\int_{\Omega}U(\bm{u})$ ]{
\begin{tikzpicture}
\begin{loglogaxis}[
        scaled ticks=false, 
        tick label style={/pgf/number format/fixed},
	legend cell align=left,
	legend style={font=\tiny},
	width=.475\textwidth,
    xlabel={CFL},
    ylabel={Average entropy},
    legend pos=south east,
    xmajorgrids=true,
    ymajorgrids=true,
    grid style=dashed,
] 
\addplot[color=blue,mark=*,semithick, mark options={solid,fill=markercolor}]
coordinates{(0.25,6.22584e-05)(0.125,2.42033e-05)(0.0625,2.70793e-06)(0.03125,9.23535e-08)};

\addplot[color=red,dashed,mark=square*,semithick, mark options={solid,fill=markercolor}]
coordinates{(0.25,5.32822e-05)(0.125,1.5792e-05)(0.0625,9.458e-07)(0.03125,3.05056e-08)};

\logLogSlopeTriangleFlip{0.275}{0.15}{0.3}{5}{black}

\legend{Affine, Curved} 
\end{loglogaxis}
\end{tikzpicture}
}
\caption{Evolution of the kinetic energy dissipation rate over time on affine and curvilinear meshes, as well as dependence of average entropy over the domain $\int_{\Omega} U(\bm{u})$ at time $T = 20$ for an entropy conservative formulation.  }
\label{fig:tg}
\end{figure}
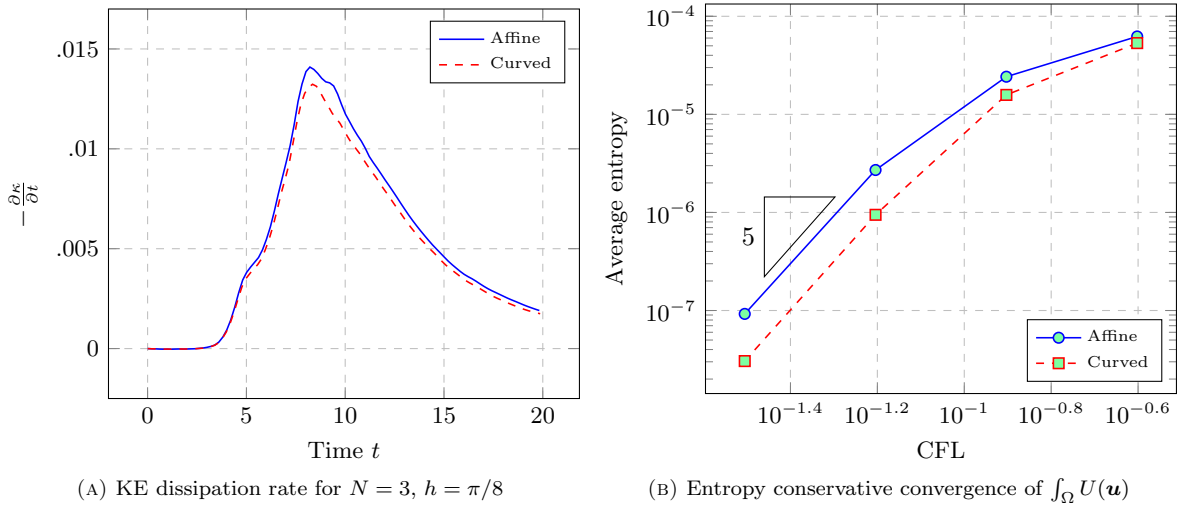


\section{Conclusions}

This work describes how to extend entropy conservative and entropy stable DG ``modal'' discretizations to curvilinear meshes using both weighted and weight-adjusted mass matrices.   Assuming that the geometric terms satisfy a discrete geometric conservation law, the presented schemes allow for the use of over-integration while satisfying a semi-discrete entropy equality or inequality on curved meshes.  Numerical results show the presented schemes achieve optimal rates of convergence for smooth solutions on both two and three dimensional affine and curvilinear meshes while remaining robust in the presence of under-resolved solutions such as shocks and turbulence.  

Several outstanding computational questions remain to be answered.  First, the use of weight-adjusted mass matrices is motivated by the low storage requirements and their efficient application on GPUs.  However, to guarantee discrete entropy stability, the computation of a weight-adjusted projection is necessary, which adds additional computational cost compared to the affine case.  Furthermore, numerical results suggest that, while computing the entropy-projected conservative variables $\tilde{\bm{u}}$ using the weight-adjusted projection is necessary to ensure a semi-discrete conservation of entropy, the difference between using a weight-adjusted projection and regular projection may be negligible in practice.  The necessity of the weight-adjusted projection in computing $\tilde{\bm{u}}$ should be examined further.  Secondly, on triangular and tetrahedral meshes, the proposed schemes involve more computational work than under-integrated collocation-style SBP schemes \cite{hicken2016multidimensional, chen2017entropy, crean2018entropy}.  A careful computational comparison of the presented schemes with existing methods should be done to weigh the benefits of improved accuracy with additional computational costs.  

Finally, we note that while this work has focused on triangular and tetrahedral elements, the approaches outlined here can be extended to more arbitrary pairings of (possibly non-polynomial) approximation spaces and quadratures.  For example, similar techniques can be used to construct entropy stable B-spline or Galerkin difference discretizations on curved meshes \cite{banks2016galerkin, chan2018multi}.  

\section{Acknowledgements}

The authors thank Mark H.\ Carpenter and David C.\ Del Rey Fernandez for informative discussions.  Jesse Chan is supported by NSF DMS-1719818 and DMS-1712639.  

\bibliographystyle{unsrturl}
\bibliography{dg}

\end{document}